\documentclass{article}

\usepackage{amsfonts}
\usepackage{amssymb}
\usepackage[english]{babel}
\usepackage{algorithm}
\usepackage{algorithmic}
\usepackage{booktabs}
\usepackage{nicefrac}
\usepackage{authblk}
\usepackage[small]{titlesec}
\usepackage{tikz,pgfplots}
\usepackage{setspace}

\usepackage[
,breaklinks=true  
,colorlinks       
,linkcolor=blue
,urlcolor=red!50!black    
,citecolor=black!40!green
,bookmarks         
,bookmarksnumbered,
pdftex=true,              
plainpages=false,      
hypertexnames=false,     
pdfpagelabels=true,   
hyperindex=true,
pdfa=true
]{hyperref}
\usepackage[leqno]{amsmath}
\usepackage{amsthm}

\newtheorem{theorem}{Theorem}[section]

\newtheorem{lemma}[theorem]{Lemma}
\newtheorem{proposition}[theorem]{Proposition}
\newtheorem{assumption}[theorem]{Assumption}

\newtheoremstyle{mystyle}
{}
{}
{\normalfont}
{}
{\normalfont\itshape}
{.}
{ }
{}

\theoremstyle{mystyle}

\newtheorem{remark}[theorem]{Remark}

\numberwithin{equation}{section}

\newcommand{\R}{\mathbb{R}}

\newcommand{\N}{\mathbb{N}}

\newcommand{\smax}{\sigma_{\max}}
\renewcommand{\d}{\text{d}}

\newcommand{\conv}{\operatorname{conv}}
\newcommand{\cone}{\operatorname{cone}}

\newcommand{\dual}[2]{\langle #1 , #2 \rangle}

\renewcommand{\AA}{\mathcal{A}}
\newcommand{\II}{\mathcal{I}}
\newcommand{\NN}{\mathcal{N}}
\newcommand{\BB}{\mathcal{B}}

\newcommand{\ie}{i.e., }
\newcommand{\eg}{e.g., }
\newcommand{\weak}{\rightharpoonup}
\newcommand{\embed}{\hookrightarrow}

\providecommand{\keywords}[1]{\textbf{Keywords.} #1}

\title{\bf Parabolic optimal control
	problems with combinatorial switching constraints \\ Part II: Outer approximation algorithm\thanks{This work has partially been supported by Deutsche Forschungsgemeinschaft (DFG) under grant no.~BU~2313/7-1 and ME~3281/10-1.}}
\author{Christoph~Buchheim, Alexandra~Gr\"utering and Christian~Meyer\footnote{\{christoph.buchheim,alexandra.gruetering,christian.meyer\}@math.tu-dortmund.de}}
\date{\vspace{-5ex}}
\affil{Department of Mathematics, TU Dortmund University, Germany}

\begin{document}
\maketitle
\renewcommand{\abstractname}{} 
\begin{abstract}
   We consider optimal control problems for partial differential
  equations where the controls take binary values but vary over the
  time horizon, they can thus be seen as dynamic switches. The
  switching patterns may be subject to combinatorial constraints such
  as, \eg an upper bound on the total number of switchings or a
  lower bound on the time between two switchings. 
  In a companion paper [\href{https://arxiv.org/abs/2203.07121}{arXiv:2203.07121}], 
  we describe the $L^p$-closure of the convex hull of feasible switching patterns
  as intersection of convex sets derived from finite-dimensional projections. 
  In this paper, the resulting outer description is used for the construction of an outer approximation algorithm 
  in function space, whose iterates are proven to converge strongly in $L^2$ to 
  the global minimizer of the convexified optimal control problem.
  The linear-quadratic subproblems arising in each iteration of the outer approximation algorithm 
  are solved by means of a semi-smooth Newton method. 
  A numerical example in two spatial dimensions illustrates the efficiency of the overall algorithm.\\
  
\end{abstract}

\section{Introduction} \label{sec: intro}

This paper is concerned with the design of an outer approximation algorithm for tailored convex relaxations 
of parabolic optimal control problems with combinatorial switching constraints, of the 
following form:
\begin{equation}\tag{P}\label{eq:optprob}
\left\{\quad
\begin{aligned}
\text{min} \quad & J(y,u) = \tfrac{1}{2}\, \|y - y_{\textup{d}}\|_{L^2(Q)}^2 + \tfrac{\alpha}{2}\,\|u-\tfrac 12\|^2_{L^2(0,T; \R^n)}\\
\text{s.t.} \quad & 
\begin{aligned}[t]
\partial_t y(t,x) - \Delta y(t,x) &= \sum_{j=1}^n u_j(t) \,\psi_j(x) & & \text{in } Q := \Omega \times (0,T),\\
y(t,x) &= 0 & & \text{on } \Gamma := \partial\Omega \times (0,T),\\  
y(0,x) &= y_0(x) & & \text{in } \Omega,
\end{aligned}\\
\text{and} \quad & u \in D.
\end{aligned}
\quad \right.
\end{equation}
Herein, $T > 0$ is a given final time and $\Omega\subset \R^d$, $d\in
\N$, is a bounded domain, \ie a bounded, open, and connected set,  with Lipschitz boundary $\partial \Omega$ in the sense of
\cite[Def. 1.2.2.1]{GRIS85}. The
form functions $\psi_j\in H^{-1}(\Omega)$, $j=1,\dots,n$, as well as
the initial state~$y_0 \in L^2(\Omega)$ are given. Moreover, $y_{\textup{d}} \in L^2(Q)$ is a given desired state
and $\alpha \geq 0$ is a Tikhonov parameter weighting the mean
deviation from~$\tfrac 12$. Finally, 
\[
D \subset \big\{ u \in BV(0,T;\R^n)\colon u(t) \in \{0,1\}^n \text{ f.a.a.\ } t \in (0,T)\big\}
\]
denotes the set of feasible \emph{switching controls}. The precise assumptions on $D$ 
and examples for such a set are given in Section~\ref{sec: pre} below. 

Problems of type~\eqref{eq:optprob} arise in various applications such as shifting of gear-switches in automotive
engineering or valves in gas and water networks, see~\eg~\cite{GER05,FUEG09,KSB10, SBF13,HAN20}.
Thus, there exists a variety of different approaches for their numerical solution and we give a brief overview 
without claiming to be exhaustive. 
One approach is to discretize the optimal control problem, which typically leads to a large-scale mixed-integer 
nonlinear  programming problem that is then solved by standard methods, 
see, \eg~\cite{GER05, Lee12, BKLL13, GPRS19, SH20}. 
Other methods employ convexifications accompanied by subsequent tailored rounding strategies. 
We exemplarily refer to~\cite{SA05,SA12,HS13,JUN14,KML17,MAN19,KLM20,ZS20} 
and the references therein.
If $D$ imposes a bound on the $BV$-seminorm (as will also be the case in our setting), then 
the elements in~$D$ are determined by a finite number of switching times. 
Several approaches rely on this observation and aim at optimizing these switching times, 
see \eg~\cite{GER06,EWA06,JM11,FMO13,HR16,ROL17,SOG17}.
In yet another class of methods, non-smooth penalty techniques, partly in combination with 
convexification, are used to account for the switching constraints in $D$. 
We refer to~\cite{CK14,CK16,CTW18,CKK18,Wac19} and the references therein.

In our companion paper~\cite{partI}, we propose an entirely different approach for the solution 
of~\eqref{eq:optprob}. This approach is based on a tailored convexification of~\eqref{eq:optprob}, 
which is built by cutting planes derived from finite-dimensional projections.
The numerical experiments in~\cite{partI} demonstrate that our convexification generally provides better dual bounds 
than the naive relaxation, which is obtained by replacing the binarity constraints $u\in \{0,1\}^n$ by $u\in [0,1]^n$ in the definition of~$D$. 
Even more, the prototypical example in~\cite[Counterexample~3.1]{partI} shows that the naive relaxation does not 
give the closure of the convex hull of $D$ in any $L^p(0,T;\R^n)$ in general.
In addition, the naive relaxation may not benefit from the particular problem structure, as the 
the investigation of the min-up/min-down polytope in~\cite{JRS15} demonstrates.
The reason is that the number of facets depends heavily on the discretization as shown by~\cite{KMS11}.

By contrast, our approach provably generates the closure of the convex hull in the limit.
To be more precise, it is shown in~\cite{partI} 
that $\overline{\conv (D)} = \bigcap_{k\in \N} V_k$, with the closure taken in~$L^p(0,T;\R^n)$, see Theorem~\ref{thm: convD} below.
Herein, the sets~$V_k$ correspond to sets in $\R^{M_k}$ that, for prominent examples, can be shown to be polytopes
for which the separation problem is tractable; see~\cite[Section~3.1 and 3.2]{partI}.
In this paper, we will combine the separation algorithm for $V_k$ with the passage to the limit $k\to \infty$ 
in order to obtain an \emph{outer approximation algorithm} whose 
iterates converge strongly in $L^2(0,T;\R^n)$ to the global minimizer of~\eqref{eq:optprob} 
with $D$ replaced by $\overline{\conv (D)}$.
Outer approximation algorithms are well-established methods for the solution of 
mixed-integer nonlinear programs, see, \eg the classical references~\cite{DG86, FL94}, 
and have also proven to work for combinatorial optimal control problems involving PDEs~\cite{BKM18}.
It is to be underlined that the sets $V_k$ are designed by means of finite dimensional projections of the 
control variable only, without any discretization of the PDE in~\eqref{eq:optprob}. 
Thus, by addressing the PDE in function space, we avoid the curse of
dimensionality caused by the widely used first-discretize-then-optimize approach. 

The plan of the paper reads as follows: 
In Section~\ref{sec: pre}, we state the standing assumptions on $D$ and recall the main results of our companion 
paper~\cite{partI} that will be needed for the construction and the analysis of our outer approximation algorithm. 
Section~\ref{sec: outerapprox} is devoted to the convergence analysis of the outer approximation algorithm 
for $k\to \infty$. In each iteration, a linear-quadatic optimal control problem subject to additional inequality constraints
is solved. This is done by means of a semi-smooth Newton method in function space presented in 
Section~\ref{sec: numsol}. 
The performance of the overall algorithm is tested in Section~\ref{sec: casestudy} based on a finite element 
discretization of a prototypical optimal control problem.
Finally, Section~\ref{sec: appendix} is dedicated to an existence result for Lagrange multipliers 
required for the design of the semi-smooth Newton method.

\section{Preliminaries} \label{sec: pre}
The main objective of this paper is to develop an efficient solution approach for the convexification of~\eqref{eq:optprob} given by
\begin{equation}\tag{PC}\label{eq:PC}
\left\{\quad
\begin{aligned}
\text{min} \quad & f(u):= J(Su,u)\\
\text{s.t.} \quad & u \in \overline{\conv D}^{L^p(0,T;\R^n)}\;.
\end{aligned}
\qquad\right.
\end{equation}
Hereby, $S\colon
L^2(0,T;\R^n) \to W(0,T):=H^1(0,T;H^{-1}(\Omega)) \cap L^2(0,T; H^1_0(\Omega))$ denotes the solution operator associated with the PDE in~\eqref{eq:optprob}, which admits for every control function~$u\in D\subset L^\infty(0,T;\R^n)$ a unique weak solution; see, \eg
\cite[Chapter~3]{Troe10}. Note that the objective function $f\colon L^2(0,T;\R^n) \to \R$ is
weakly lower semi-continuous because the mappings $u \mapsto \|Su -
y_{\textup{d}}\|_{L^2(Q)}^2$ and~$u \mapsto \|u-\tfrac 12 \|_{L^2(0,T;\R^n)}^2$
are both convex and lower semi-continuous, thus weakly lower
semi-continuous, and the solution operator $S$ is affine and
continuous, thus weakly continuous.
The
set \[D\subset \big\{ u \in BV(0,T;\R^n)\colon u(t)
\in \{0,1\}^n \text{ f.a.a.\ } t \in (0,T)\big\}\] of feasible switching controls is supposed to satisfy the two
following assumptions:
\begin{align}
& \text{$D$ is a bounded set in $BV(0,T;\R^n)$,} \tag{D1}\label{eq:D1} \\
& \text{$D$ is closed in $L^p(0,T;\R^n)$ for some fixed $p \in [1,\infty)$,} \tag{D2}\label{eq:D2}
\end{align}
where $BV(0,T;\R^n)$ denotes the set of all vector-valued functions with bounded variation, \ie
\[BV(0,T;\R^n):=\{u\in L^1(0,T;\R^n): u_i\in BV(0,T) \text{ for } i=1,\ldots,n\,
\}\] equipped with the norm
\[\|u\|_{BV(0,T;\R^n)} := \|u\|_{L^1(0,T;\R^n)} + \sum_{j=1}^n |u_j|_{BV(0,T)}\;.\]
For controls $u\in D$ the BV-seminorm~$|u_j|_{BV(0,T)}$ agrees with the minimal number of switchings of any representative of~$u_j$ with values
in~$\{0,1\}$. More details on the space of bounded variation functions can be found in
\cite[Chap.~10]{ATT14}. 

The set $D$ covers a wide range of combinatorial switching constraints, for instance, in the case of an upper bound on the total number of switchings the set is given as 
\begin{equation}\label{eq:Dex}
\begin{aligned}
D_{\max} := \big\{ u \in BV(0,T;\R^n)\colon \;
& u(t) \in \{0,1\}^n \text{ f.a.a.\ } t \in (0,T),\\
& |u_j|_{BV(0,T)} \leq \smax\; \forall \,j = 1, \dots, n \big\},
\end{aligned}
\end{equation}
where $\smax\in \N$ is a given number. 

Under the assumptions~\eqref{eq:D1} and~\eqref{eq:D2} on the set $D$, the control problem~\eqref{eq:optprob} admits a global minimizer and the convex hull of the feasible switching patterns can be fully described by cutting planes lifted from finite-dimensional projections; see~\cite{partI}.
For the latter, the set~$D$ is projected, by means of \begin{equation}\label{eq:pi}
\Pi\colon BV(0,T;\R^n) \ni u \mapsto \big(\langle \Phi_i, u
\rangle\big)_{i=1}^M \in \R^{M}\;,
\end{equation} to the finite-dimensional space~$\R^M$, where $\Phi_i \in
L^p(0,T;\R^n)^*$, $i=1, \dots, M$, are linear and continuous functionals, \eg local averaging operators of the
form \begin{equation}\label{eq:localaveraging} \langle \Phi_{(j-1)N+i}, u
\rangle := \tfrac{1}{\lambda(I_{i})}\int_{I_{i}} u_{j}\, \d t
\end{equation}
for $j=1,\ldots,n$ with suitably chosen subintervals $I_i\subset (0,T)$, $i=1,\ldots,N$, and $M:=n\,N$. 
Each projection~$\Pi$ then gives rise to a
relaxation of the closed convex hull of the set~$D$ in~$L^p(0,T;\R^n)$, which we will use to derive outer approximations by linear inequalities.
\begin{lemma}[{\cite[Lemma~3.2]{partI}}]\label{lem: convD}
	For any $\Pi$ as in~\eqref{eq:pi}, we have
	\[\overline{\conv (D)}^{ L^p(0,T;\R^n)}\subseteq \{ v \in
	L^p(0,T;\R^n)\colon \Pi(v) \in C_{D,\Pi}\}\;,\]
	where 
	\[
	C_{D,\Pi} := \conv\{\Pi (u)\colon u \in D \}\subset \R^{M}\;.
	\]
\end{lemma}
Note that, based on the general assumptions~\eqref{eq:D1} and~\eqref{eq:D2}, it is easy to see that the finite dimensional set $C_{D,\Pi}$ is closed in $\R^M$ for any projection and consequently, the set $\{ v \in
L^p(0,T;\R^n)\colon \Pi(v) \in C_{D,\Pi}\}$ is convex and closed in~$L^p(0,T;\R^n)$.   

In addition, projections~$\Pi_k$, for increasing~$k$, can be designed in such a way that an outer description of all finite-dimensional convex hulls $C_{D,\Pi_k}$
also leads to an outer description of the convex hull of
$D$ in function space.

\begin{theorem}[{\cite[Thm.~3.5]{partI}}]\label{thm: convD}
	For each~$k\in\N$, let~$I^k_1,\dots,I_{N_k}^k$, $N_k\in\N$, be
	disjoint open intervals in~$(0,T)$ such that
	\begin{itemize}
		\item[(i)] $\bigcup_{i=1}^{N_k} \overline{I_i^k} = [0,T]$ for all~$k\in\N$,
		\item[(ii)] $\max_{i=1,\dots,N_k}\lambda(I_i^k)\to 0\; \mbox{ for } k\to \infty$, and
		\item[(iii)] for each~$r\in\{1,\dots,N_{k+1}\}$ there
		exists~$i\in\{1,\dots,N_{k}\}$ such that~$I^{k+1}_{r}\subseteq
		I^k_i$, \ie the intervals form a nested sequence. 
	\end{itemize}
	Set $M_k := n\, N_k$ and define projections~$\Pi_{k}\colon BV(0,T;\R^n) \to \R^{M_k}$,
	for~$k\in\N$, by 
	\begin{equation}\label{eq:Pik}
	\langle\Phi_{(j-1)N_k+i}^k,u\rangle :=
	\tfrac{1}{\lambda (I_i^k)}\int_{I_i^k} u_j(t)\, \d t\;
	\end{equation}
	for $j=1,\ldots,n$ and $i=1,\ldots,N_k$. Moreover, set
	\[V_k:=\{ v \in L^p(0,T;\R^n)\colon \Pi_{k}(v) \in
	C_{D,\Pi_k}\}\;.\]
	Then
	$V_{k}\supseteq V_{k+1}$ for all $k\in\N$ and 
	\begin{equation}\label{eq:convD}
	\overline{\conv (D)}^{ L^p(0,T;\R^n)} = \bigcap_{k\in \N} V_k\;.
	\end{equation} 
	
\end{theorem}
Compared to Theorem~3.5 in~\cite{partI}, the assumption (iii) in the above theorem is additional. It is easy to see that it guarantees $V_k\supseteq V_{k+1}$ for all $k\in\N$, considering that each entry of~$\Pi_k$ is a convex combination of entries of~$\Pi_{k+1}$. The second assertion~\eqref{eq:convD} has been proven in~\cite{partI}.

\section{Outer approximation algorithm} \label{sec: outerapprox}

In the following, we explain how to address the convexified problem~\eqref{eq:PC}
by an outer approximation approach. We use the outer descriptions of the sets $C_{D,\Pi}$ appearing in~\eqref{eq:convD} to cut off any control $u\in L^p(0,T;\R^n)$ violating some of the conditions $\Pi(u)\in C_{D,\Pi}$.

More formally, we first fix an operator $\Pi\colon BV(0,T;\R^n) \ni u
\mapsto \big(\langle \Phi_i, u \rangle\big)_{i=1}^M \in \R^{M}$ such that $\Pi(u)\notin C_{D,\Pi}$ holds. Since the convex set~$C_{D,\Pi}$
is closed in $\R^M$, it is the intersection of its
supporting half spaces and can be described by linear inequality
constraints. The number of necessary half
spaces can be infinite in general~\cite[Ex.~3.6]{partI}, but for many practically relevant
constraints~$D$, it turns out to be finite; see~\cite[Sect.~3.1 and~3.2]{partI}.
Let us define the set of
all valid linear inequalities for $C_{D,\Pi}$
as
\[
H_{D,\Pi}=\{(a,b)\in [-1,1]^M\times \R: a^\top w \leq b \ \forall w\in C_{D,\Pi}\}\;,
\]
where $a\in[-1,1]^M$ can be assumed without loss of generality by scaling. 
To cut off the infeasible control~$u$, we choose a violated linear inequality constraint and add this constraint to the problem.
For the rest of this section, we assume that the local averaging operators satisfy the conditions (i)--(iii)  of Theorem~\ref{thm: convD}.
Our outer approximation algorithm for~\eqref{eq:PC} then reads as follows:

\begin{algorithm}[H]
	\caption{Outer approximation algorithm for~\eqref{eq:PC}}\label{alg:PC}
	\begin{algorithmic}[1]
		\STATE{Set $k=0$, $T_0 = \emptyset$, $I_1^0=(0,T)$  and $N_0=1$.}
		\STATE{\label{it:step2} Solve 
			\begin{equation}\tag{PC\text{$_k$}}\label{eq:PCk}
			\left\{\quad		    
			\begin{aligned}
			\min\quad &f(u) \\ 
			\text{s.t.} \quad & u \in [0,1]^n \quad\text{a.e.~in }(0,T), \\
			& a^\top \Pi(u) \leq b \quad\forall \, (\Pi, a, b) \in T_k\;.
			\end{aligned}
			\qquad\right.
			\end{equation} 
			Let $u^k$ be the optimal solution.} 
		\IF{$u^k \in \overline{\conv D}^{L^p(0,T;\R^n)}$\label{it:step3}}
		\RETURN $u^k$ as optimal solution. 
		\ELSE
		\STATE{\label{it:step6}
			Determine intervals $I_i^{k+1}$, $1\leq i\leq N_{k+1}$, 
				such that $\Pi_{k+1}(u^k)\notin C_{D,\Pi_{k+1}}$.}
		\STATE{\label{it:step7}
			Find an optimizer $(a_{k+1}, b_{k+1})\in {\arg \max}_{(a,b)\in H_{D,\Pi_{k+1}}} (a^\top \Pi_{k+1}(u^k) -b)$.} 
		\STATE{Set $T_{k+1}=T_{k}\cup \{(\Pi_{k+1}, a_{k+1}, b_{k+1})\}$, $k = k+1$ and go to~\ref{it:step2}.}
		\ENDIF
	\end{algorithmic}
\end{algorithm}
Some remarks on Algorithm~\ref{alg:PC} are in order. First note that, by the standard direct method of calculus of variations, one can easily show the existence of a global minimizer for~\eqref{eq:PCk} and its uniqueness if the Tikhonov parameter $\alpha$ is positive. Step
\ref{it:step7} of the algorithm is well defined since
$C_{D,\Pi_{k+1}}\neq \emptyset$ and hence~$b$ is bounded from
below. Moreover, Step~\ref{it:step6} is well defined due to
\eqref{eq:convD}.  Consequently, an important subproblem in the outer
approximation algorithm consists in determining appropriate
intervals~$I_i$ of the local averaging operators, such that for a
given $u^k$ it holds~$\Pi(u^k) \notin C_{D,\Pi}$. In view of
Theorem~\ref{thm: convD}, the desired property $\Pi(u^k)\notin~C_{D,\Pi}$
follows as soon as~$\Pi$ is defined by a large enough number of small
enough intervals,  and remains valid for all further
refinements. Note,
however, that Step~\ref{it:step6} does not exclude to
set~$\Pi_{k+1}=\Pi_k$ if this suffices to cut off~$u^k$.  Finally,
we emphasize that the stopping criterion in Step~\ref{it:step3} is rather
symbolic; in general, it can be verified only by showing that no
further violated cutting planes exist, for any
projection.

From a practical point of view, we obtain~$u^k$ by solving the
parabolic optimal control problem~\eqref{eq:PCk}, so that we know~$u^k$
only subject to a given discretization of~$(0,T)$; see Section~\ref{sec:
	numsol} for more details on the numerical solution of
\eqref{eq:PCk}. One could thus argue that the best possible approach is
to choose the intervals~$I_i$ exactly as given by this
discretization. This may be a feasible approach provided that the
finite-dimensional separation algorithm for~$C_{D,\Pi}$, needed in
Step \ref{it:step7}, is fast enough to deal with problems of large
dimension~$M$, as it is the case for a switch-wise upper bound on the
total number of shiftings as defined in~\eqref{eq:Dex}; see Section~\ref{sec:
	casestudy}. However, one cannot expect such a fast separation
algorithm for general switching constraints, so that it may be
necessary to restrict oneself to a smaller number of intervals.

We now investigate the convergence behavior of Algorithm~\ref{alg:PC}.  It
turns out that choosing the most violated inequality in
Step~\ref{it:step7} is crucial to guarantee convergence; this is a
common choice in semi-infinite programming~\cite{Gust73}.  In
addition, we have to require additional assumptions on the partitions
of $(0,T)$ used for the construction of the local averaging
operators: besides the hypotheses~(i)--(iii) from Theorem~\ref{thm: convD}, we
have to assume that the partitions are quasi-uniform. For this
purpose, we introduce
\[
\tau_k:=\min_{1\le i \le N_k} \lambda(I_i^k) \quad \mbox{ and } \quad 
h_k:=\max_{1\le i \le N_k} \lambda(I_i^k),
\] and require
\begin{assumption}\label{assu:rel}
	There exists $\kappa>0$ such that 
	$h_k\leq \kappa\,\tau_k$ for every $k\in \N$. 
\end{assumption}

Given this assumption, we can prove the following 
\begin{theorem}\label{thm:PC}
	Assume that Algorithm~\ref{alg:PC} does not stop after a finite number of
	iterations and the sequence $I_1^k,\ldots,I_{N_k}^k$ resulting from
	Step \ref{it:step6} is constructed such that it meets the
	assumptions (i)--(iii) from Theorem~\ref{thm: convD} and Assumption \ref{assu:rel}.
	Suppose in addition that the Tikhonov parameter $\alpha$ is
	positive.  Then the sequence $\{u^k\}_{k\in \N}$ converges strongly
	in $L^2(0,T;\R^n)$ to the unique global minimizer of~\eqref{eq:PC}.
\end{theorem}
\begin{proof}
	Thanks to the box constraint $u \in [0,1]^n$ a.e in $(0,T)$, the
	sequence $\{u^k\}_{k\in \N}$ is bounded in $L^\infty(0,T;\R^n)$ so
	that there exists a weakly-$\ast$ converging subsequence, denoted by $u^{k_m}\weak^* u^\star$ in
	$L^\infty(0,T;\R^n)$.  Since weak-$\ast$ convergence implies weak
	convergence in $L^p(0,T;\R^n)$ and the local averaging operators are
	clearly weakly continuous, we thus get $\Pi(u^{k_m})\to
	\Pi(u^\star)$ for $m\to\infty$ and any projection $\Pi$ occurring in
	Algorithm~\ref{alg:PC}.  Additionally, the set
	\[\{u \in L^p(0,T;\R^n): u \in [0,1]^n \text{ a.e.~in }(0,T)\}\] is
	convex and closed, hence weakly closed, and therefore
	$u^\star(t)\in[0,1]^n$ a.e. in $(0,T)$. Consequently, $u^\star$ is
	feasible for all problems~\eqref{eq:PCk}, $k\in \N$.  The optimality
	of $u^{k_m}$ for~(PC$_{k_m}$) now implies $f(u^{k_m}) \leq
	f(u^\star)$ and the weak lower semi-continuity of $f$ thus gives
	\begin{equation}\label{eq:opt}
	f(u^\star) \leq \liminf_{m\to\infty} f(u^{k_m}) \leq \limsup_{m\to\infty} f(u^{k_m}) \leq f(u^\star),
	\end{equation}
	\ie $f(u^{k_m}) \to f(u^\star)$. Since $ u \mapsto \|Su -
	y_{\textup{d}}\|_{L^2(Q)}^2$ and $u \mapsto \|u-\tfrac
	12\|_{L^2(0,T;\R^n)}^2$ are both convex and lower semi-continuous,
	thus weakly lower semi-continuous, the convergence of the objective
	and the assumption~$\alpha>0$ imply
	\[
	\|u^{k_m}-\tfrac 12\|_{L^2(0,T;\R^n)}^2 \to \|u^\star-\tfrac 12\|_{L^2(0,T;\R^n)}^2 .
	\] 
	Since weak and norm convergence in Hilbert spaces imply strong
	convergence, this gives the strong convergence of $\{u^{k_m}\}_{m\in \N}$ to $u^\star$ in
	$L^2(0,T;\R^n)$. 
	
	We next prove
	\begin{equation}\label{eq:Vl}
	u^\star \in V_\ell=\{ v \in L^p(0,T;\R^n)\colon
	\Pi_\ell(v) \in C_{D,\Pi_\ell}\} \quad \forall \ell\in \N.
	\end{equation} 
	To this end,	  
	let $\ell\in\N$ be arbitrary, but fixed, and choose 
	\[
	(\bar{a},\bar{b})\in\operatorname{argmax}_{(a,b)\in H_{D,\Pi_\ell}} (a^\top \Pi_\ell(u^\star) -b).	  
	\]
	Then we obtain for every $k \geq \ell$ and every $u\in L^p(0,T;\R^n)$ that 
	\begin{equation}\label{eq:defatilde}
	\begin{aligned}
	\bar{a}^\top \Pi_\ell(u)
	&= \sum_{j=1}^n \sum_{i = 1}^{N_\ell}
	\bar{a}_{(j-1) N_\ell + i} \,\tfrac{1}{\lambda(I_i^\ell)}\int_{I_i^\ell} u_{j}(t) \, \d t \\
	&= \sum_{j=1}^n \sum_{i=1}^{N_\ell} \bar{a}_{(j-1) N_\ell + i} \,\tfrac{1}{\lambda(I_i^\ell)} 
	\sum_{I_r^k\subseteq I_i^\ell}\int_{I_r^k} u_{j}(t) \, \d t\\ 
	&= \sum_{j=1}^n \sum_{i=1}^{N_\ell} \sum_{I_r^k\subseteq I_i^\ell} 
	\underbrace{\bar{a}_{(j-1) N_\ell + i} \,
		\tfrac{\lambda(I_r^k)}{\lambda(I_i^\ell)}}_{\displaystyle{=: (\tilde{a}_k)_{(j-1) N_k + r}}}
	\, \tfrac{1}{\lambda(I_r^k)}
	\int_{I_r^k} u_{j}(t) \, \d t 
	= \tilde{a}_k^\top \Pi_{k}(u)\;.
	\end{aligned}
	\end{equation}
	Note that the vector $\tilde a^k = ((\tilde a_k)_1, \ldots , (\tilde
	a_k)_{M_k}) \in \R^{M_k}$, $M_k = n \, N_k$, is well defined, since
	the intervals are nested by assumption (iii) in Theorem~\ref{thm: convD}.
	Thus the convergence of~$u^{k_m}$ to~$u^\star$ yields
	\begin{equation}\label{eq:cut}
	\begin{aligned}
	\bar{a}^\top \Pi_\ell(u^\star)-\bar{b}
	&=	\lim\limits_{m\to \infty} \bar{a}^\top \Pi_\ell(u^{k_m})-\bar{b}\\
	&=\lim\limits_{m\to \infty} \tilde{a}_{k_m+1}^\top \Pi_{k_m+1}(u^{k_m})-\bar{b} \\
	&=\lim\limits_{m\to \infty} \tfrac{h_{k_m+1}}{\tau_\ell}
	\left[ \tfrac{\tau_\ell}{h_{k_m+1}}\big(\tilde{a}_{k_m+1}^\top \Pi_{k_m+1}(u^{k_m})-\bar{b}\big) \right].
	\end{aligned}
	\end{equation}
	Moreover, for every $u \in D$ and every $k\geq \ell$, we deduce from
	\eqref{eq:defatilde} and $(\bar a, \bar b)\in H_{D,\Pi_\ell}$ that
	$\tilde a_k^\top\Pi_k(u) = \bar a^\top \Pi_\ell(u) \leq \bar b$,
	such that $(\tilde{a}_{k},\bar b)$ induces a valid inequality for
	$C_{D,\Pi_{k}}$.  Hence, for $k$ sufficiently large,
	$\tfrac{\tau_\ell}{h_{k}}(\tilde{a}_{k},\bar b)$ induces a valid
	inequality as well, where the coefficients satisfy
	\[
	\tfrac{\tau_\ell}{h_{k}} \, |(\tilde{a}_k)_{(j-1) N_{k} + r}|
	= \tfrac{\tau_\ell}{\lambda(I_i^\ell)} \, \tfrac{\lambda(I_r^k)}{h_{k}} \,
	|\bar{a}_{(j-1) N_\ell + i}| 
	\leq  |\bar{a}_{(j-1) N_\ell + i}| \leq 1
	\]
	for all $j = 1, \ldots, n$ and all $r = 1, \ldots, N_k$.
	Thus $\tfrac{\tau_\ell}{h_{k_m+1}}(\tilde{a}_{k_m+1},\bar
	b)\in H_{D,\Pi_{k_m+1}}$, provided that $m$ is sufficiently large,
	which in turn gives
	\[
	\tfrac{\tau_\ell}{h_{k_m+1}}\big(\tilde{a}_{k_m+1}^\top \Pi_{k_m+1}(u^{k_m})-\bar{b}\big)
	\leq a_{k_m+1}^\top \Pi_{k_m+1}(u^{k_m}) - b_{k_m+1},
	\]
	because the most violated cutting plane is chosen in Step
	\ref{it:step7} of Algorithm~\ref{alg:PC}.  Together with~\eqref{eq:cut}, the latter
	yields
	\begin{equation}\label{eq:ineq_cut}
	\begin{aligned}
	\bar{a}^\top \Pi_\ell(u^\star)-\bar{b}
	\leq \tfrac{1}{\tau_\ell}\,\liminf_{m\to \infty}\ h_{k_m+1}( a_{k_m+1}^\top \Pi_{k_m+1}(u^{k_m}) - b_{k_m+1}).
	\end{aligned}
	\end{equation}
	Since $u^\star$ is feasible for all~\eqref{eq:PCk} as seen above, we obtain for the right hand side
	\[
	\begin{aligned}
	& h_{k_m+1}\, ( a_{k_m+1}^\top \Pi_{k_m+1}(u^{k_m}) - b_{k_m+1})\\
	&\quad =  h_{k_m+1}\,(a_{k_m+1}^\top \Pi_{k_m+1}(u^\star)-b_{k_m+1})
	+ h_{k_m+1}\, a_{k_m+1}^\top\Pi_{k_m+1}(u^{k_m}-u^\star)\\
	&\quad \leq h_{k_m+1}\, a_{k_m+1}^\top\Pi_{k_m+1}(u^{k_m}-u^\star)
	\end{aligned}
	\] 
	and, since $a_{k_m+1} \in [-1,1]^{M_{k_m+1}}$, we can further estimate
	\begin{equation}\label{eq:ukmconv}
	\begin{aligned}
	& |h_{k_m+1}\, a_{k_m+1}^\top\Pi_{k_m+1}(u^{k_m}-u^\star)| \\
	&\qquad \leq h_{k_m+1} \sum_{j=1}^n \sum_{i=1}^{N_{k_m+1}} \tfrac{1}{\lambda(I_i^{k_m+1})} 
	\int_{I_i^{k_m+1}} |u^{k_m}_j - u^\star_j| \, \d t \\
	& \qquad \leq \tfrac{h_{k_m+1}}{\tau_{k_m+1}} \sum_{j=1}^n \sum_{i=1}^{N_{k_m+1}} 
	\int_{I_i^{k_m+1}} |u^{k_m}_j - u^\star_j| \, \d t \\
	&\qquad \leq \kappa \sum_{j=1}^n \| u_j^{k_m} - u^\star_j\|_{L^1(0,T)} \to 0, 
	\quad\text{as } m \to \infty,
	\end{aligned}
	\end{equation}       
	where we used Assumption~\ref{assu:rel} and the strong convergence of $u^{k_m}$ to $u^\star$.
	From~\eqref{eq:ineq_cut} we now obtain $\bar{a}^\top \Pi_\ell(u^\star)-\bar{b}\leq 0$ 
	and thus $a^\top \Pi_\ell(u^\star) -b\leq 0$ for all $(a,b)\in  H_{D,\Pi_\ell}$ 
	due to the choice $(\bar{a},\bar{b})\in\arg \max_{(a,b)\in H_{D,\Pi_\ell}} (a^\top
	\Pi_\ell(u^\star) -b)$. This gives $u^\star\in V_\ell$, as claimed. 
	
	Since $\ell\in \N$ was arbitrary, we finally arrive at
	\[
	u^\star
	\in \bigcap_{\ell\in \N} V_\ell=\overline{\conv
		D}^{L^p(0,T;\R^n)}\;,
	\]
	where the equality was shown in Theorem~\ref{thm: convD}, \ie 
	$u^\star$ is feasible for~\eqref{eq:PC}.
	To show optimality, consider any $u \in L^p(0,T;\R^n)$ feasible
	for~\eqref{eq:PC}. Then~$u$ is also feasible
	for~(PC$_{k_m}$) for every $m\in \N$, and the optimality of
	$u^{k_m}$ implies $f(u^{k_m}) \leq f(u)$. 
	Due to $f(u^{k_m})\to f(u^\star)$ by~\eqref{eq:opt}, we thus have the optimality of $u^\star$.
	
	Now, since $\alpha > 0$ by assumption,\eqref{eq:PC} is a strictly
	convex problem such that $u^\star$ is the unique global minimizer of
	\eqref{eq:PC}. A well-known argument by contradiction then shows the
	strong convergence of the whole sequence $\{u^k\}_{k\in\N}$.
\end{proof}

\begin{remark}
	An inspection of the above proof allows the following modification
	of the quasi-uniformity condition in Assumption~\ref{assu:rel}: since the
	subsequence $\{u^{k_m}\}_{m\in \N}$ is bounded in
	$L^\infty(0,T;\R^n)$, Lebesgue's dominated convergence theorem gives
	that~$u^{k_m}$ converges strongly to $u^\star$ in $L^q(0,T;\R^n)$
	for every $q<\infty$. With an estimate analogous to
	\eqref{eq:ukmconv} and H\"older's inequality, one then sees that the
	condition
	\begin{equation}\label{eq:genquasiuni}
	\sum_{i=1}^{N_k} h_k^{q'} \,\lambda(I_i^{k})^{1 - q'} \leq C < \infty
	\quad \text{for all } k \in \N
	\end{equation}
	is sufficient for the convergence result in
	\eqref{eq:ukmconv}. Herein, $q'$ is the conjugate exponent and can
	thus be chosen arbitrarily close to $1$. It is easily seen that
	Assumption~\ref{assu:rel} implies~\eqref{eq:genquasiuni}. Nevertheless, we
	decided to require the stronger Assumption~\ref{assu:rel}, since it is more
	elementary and certainly more relevant from a practical point of
	view.
\end{remark}

\section{Solution of OCP-relaxations} \label{sec: numsol}

It remains to explain how we solve the optimal control problems
\eqref{eq:PCk} appearing in the outer approximation algorithm
numerically. We first set down the KKT-condition for~\eqref{eq:PCk}.
For this purpose, we introduce the linear and continuous (and thus
Fr\'echet differentiable) operator
\[
\Psi\colon L^2(0,T;\R^n) \to L^2(0,T;H^{-1}(\Omega)), \quad 
(\Psi u) (t) = \sum_{j=1}^n u_j(t) \psi_j
\] 
as well as the solution operator $\Sigma : L^2(0,T;H^{-1}(\Omega)) \to
W(0,T)$ of the heat equation with homogeneous initial condition, i.e.,
given $w\in L^2(0,T;H^{-1}(\Omega))$, $y = \Sigma(w)$ solves
\[
\partial_t y - \Delta y = w \quad \text{in } L^2(0,T;H^{-1}(\Omega)), 
\quad y(0) = 0 \quad\text{in } L^2(\Omega).
\]
Moreover, we introduce the function $\zeta \in W(0,T)$ as solution of
\[
\partial_t \zeta - \Delta \zeta = 0 \quad \text{in } L^2(0,T;H^{-1}(\Omega)), 
\quad \zeta(0) = y_0 \quad\text{in } L^2(\Omega).
\]
The solution mapping $S\colon u \mapsto y$ in~Section~\ref{sec: pre} is then given by $ S = \Sigma \circ \Psi + \zeta$.  In the
following, we will consider $S$, $\Sigma$, and $\Psi$ with different
domains and ranges.  With a little abuse of notation, we will always
use the same symbols.

With $\Sigma$ and $\Psi$ at hand, the reduced objective in
\eqref{eq:PCk} reads
\[
f(u) = \tfrac{1}{2}\, \|\Sigma \Psi u + \zeta - y_{\textup{d}}\|_{L^2(Q)}^2 
+ \tfrac{\alpha}{2} \,\|u - \tfrac{1}{2}\|_{L^2(0,T;\R^n)}^2
\]
such that, by the chain rule, its Fr\'echet derivative at $u \in
L^2(0,T;\R^n)$ is given by
\begin{equation}\label{eq:derivobj}
f'(u) = \Psi^* \Sigma^* (\Sigma \Psi u + \zeta - y_{\textup{d}}) + \alpha (u - \tfrac{1}{2}) \in L^2(0,T;\R^n),
\end{equation}
where we identified $L^2(0,T;\R^n)$ with its dual using the Riesz
representation theorem.  By standard methods,
see \eg~\cite[Sect.~3.6]{Troe10}, one shows that the adjoint $\pi = \Sigma^* g$, for given $g\in
L^2(0,T;H^{-1}(\Omega)) \embed W(0,T)^*$,  is the
solution of the backward-in-time problem
\begin{equation}\label{eq:adjPDE}
- \partial_t \pi - \Delta \pi = g \quad \text{in } L^2(0,T;H^{-1}(\Omega)), 
\quad \pi(T) = 0 \quad\text{in } L^2(\Omega)
\end{equation}
and is therefore an element of $W(0,T)$, i.e., $\Sigma^* : L^2(0,T;H^{-1}(\Omega)) \to W(0,T)$ 
is the solution operator of~\eqref{eq:adjPDE}.
Furthermore, the adjoint of $\Psi$ is given by
\[
\begin{aligned}
& \Psi^* : L^2(0,T;H_0^1(\Omega)) \to L^2(0,T;\R^n), \\
& (\Psi^* w)(t) = \Big(\dual{\psi_j}{w(t)}_{H^{-1}(\Omega), H^1_0(\Omega)}\Big)_{j=1}^n
\quad \text{f.a.a.\ } t \in (0,T).
\end{aligned}
\]
Now we have everything at hand to apply the results of Appendix~\ref{sec: appendix} 
to obtain the following KKT conditions:

\begin{proposition}
	Denote the inequality constraints associated with the cutting planes in~\eqref{eq:PCk} by 
	$G u \leq b$ with $G: L^p(0,T;\R^n) \to \R^k$ and $b\in \R^k$.
	Assume moreover that a function $\hat u \in L^\infty(0,T;\R^n)$ and a number $\delta > 0$ exist such that 
	\begin{align}
	& \delta \leq \hat u_i(t) \leq  1 - \delta \quad \text{for all } i = 1, \ldots, n 
	\text{ and f.a.a.\ } t\in (0,T) ,\label{eq:slater1}\\
	&G\hat u \leq b. \label{eq:slater2}
	\end{align}
	Then a function $\bar u \in L^\infty(0,T;\R^n)$ with associated state $\bar y = S(\bar u) \in W(0,T)$ 
	is optimal for~\eqref{eq:PCk} if and only if Lagrange multipliers $\lambda \in \R^k$ 
	and $\mu_a, \mu_b \in L^2(0,T;\R^n)$ and an adjoint state $p\in W(0,T)$ exist such that 
	the following optimality system is fulfilled:
	\begin{gather}
	- \partial_t p - \Delta p = \bar y - y_{\textup{d}} \quad \text{in } L^2(0,T;H^{-1}(\Omega)), 
	\quad p(T) = 0 \quad\text{in } L^2(\Omega),\label{eq:adeq}\\
	\Psi^*p + \alpha\, (\bar u - \tfrac{1}{2}) + \mu_b - \mu_a 
	+G^* \lambda = 0 \quad \text{a.e.~in }(0,T), \label{eq:grad}\\
	\mu_a \geq 0, \quad\quad  \mu_a \bar u  = 0, 
	\quad\quad \bar u \geq 0 \quad \text{a.e.~in }(0,T), \label{eq:comp_ua} \\
	\mu_b \geq 0, \quad \mu_b (\bar u - 1) = 0, 
	\quad \bar u \leq 1 \quad \text{a.e.~in }(0,T),\label{eq:comp_ub}\\
	\lambda \geq 0, \quad \lambda^\top (G\bar u - b) = 0,\quad G\bar u \leq b\;. \label{eq:comp_cut}
	\end{gather}
\end{proposition}

\begin{proof}
	In view of~\eqref{eq:derivobj} and~\eqref{eq:adjPDE}, 
	the necessity of ~\eqref{eq:adeq}--\eqref{eq:comp_cut} immediately follows from Theorem~\ref{thm:lagrange}.
	Due to the convexity of the optimal control problem~\eqref{eq:PCk}, 
	these conditions are also sufficient for (global) optimality. 
\end{proof}
It is easily verified that the Slater condition~\eqref{eq:slater1} and
\eqref{eq:slater2} is satisfied when the switches must additionally satisfy certain combinatorial conditions at any point in time~\cite[Sect.~3.1]{partI} or in the presence of linear constraints on the switching points~\cite[Sect.~3.2]{partI}, \eg with $u\equiv 1/2$. Consequently, in most of the practically relevant classes of constraints~$D$ the Slater conditions are fulfilled.

The pointwise resp.\ componentwise complementarity systems can equivalently be expressed by 
nonlinear complementarity functions such as, \eg the $\max$- or $\min$-function, 
which leads to the following equivalent system to~\eqref{eq:grad}--\eqref{eq:comp_cut}:
\[
\begin{aligned}
&\Psi^*p + \alpha (\bar u - \tfrac{1}{2}) + G^* \lambda\\[-1ex]
&\qquad+ \min\Big(-\Psi^*p - G^* \lambda  + \tfrac{\alpha}{2}, 0\Big) \\[-1ex]
&\qquad+ \max\Big(-\Psi^*p - G^* \lambda - \tfrac{\alpha}{2}, 0\Big) = 0
\quad\text{a.e.~in } (0,T),\\
&\rho \lambda + \max(0,G\bar u+\rho \lambda - b)=0,
\end{aligned}
\]
where $\rho>0$ can be chosen arbitrarily.
Herein, we use the same symbol for the componentwise mapping $\R^k \ni v \mapsto (\max(v_i, 0))_{i=1}^k \in \R^k$ 
and the $\max$-operator in function space.
In view of $p = \Sigma^* (\Sigma\Psi \bar u + \zeta - y_{\textup{d}})$ the optimality system is thus equivalent to 
$F(\bar u, \lambda) = 0$ with $F\colon L^2(0,T;\R^n)\times \R^k \to L^2(0,T;\R^n)\times \R^k$ defined by 
\begin{equation}\label{eq:F1}
\begin{aligned}
F_1(u,\lambda) := \;
& \Psi^*\Sigma^*(\Sigma \Psi u + \zeta - y_\textup{d})+ \alpha(u - \tfrac{1}{2}) 
+  G^* \lambda \\[-1ex]
&\quad + \min\Big(-\Psi^*\Sigma^*(\Sigma\Psi u + \zeta -y_\textup{d}) - G^* \lambda
+ \tfrac{\alpha}{2},0 \Big) \\[-1ex]
&\quad + \max\Big(-\Psi^*\Sigma^*(\Sigma\Psi u + \zeta -y_\textup{d}) - G^* \lambda 
-\tfrac{\alpha}{2}, 0\Big)
\end{aligned}
\end{equation} and 
\begin{equation}
F_2(u,\lambda) = 
-\rho \lambda + \max(0,Gu+\rho \lambda - b) .\label{eq:F2}
\end{equation}
We now use the concept of semi-smoothness as developed in~\cite{CH00}, see also the work of~\cite{HIK02}, 
to solve the above optimality system by means of a semi-smooth Newton method. For this purpose, we need 
the following

\begin{assumption}\label{assu:newtondiff}
	In addition to our standing assumptions, there are exponents $q > 2$ and $0 < s < 2/q$ such that 
	the form functions satisfy $\psi_j \in H^{s}_0(\Omega)^*$, $j = 1, \ldots, n$, and the linear functionals 
	from~\eqref{eq:localaveraging}
	fulfill $\Phi_i\in L^{q'}(0,T,\R^n)^*$, $i = 1, \ldots, M$, where $q'$ is the conjugate exponent, 
	i.e., $1/q + 1/q' = 1$.
\end{assumption}
Note that this mild additional regularity assumption on the functionals $\Phi_i$ is satisfied by the local averaging operators considered throughout this paper.

\begin{lemma}\label{lem:newton}
	Under~Assumption~\ref{assu:newtondiff}, 
	the function $F$ 
	given by~\eqref{eq:F1} and~\eqref{eq:F2} is Newton (or slant) differentiable.
\end{lemma}

\begin{proof}
	The proof is standard, but for convenience of the reader, we sketch the arguments.
	The operator $\Pi$ is linear and continuous with respect to $u$ such that 
	\[
	L^2(0,T;\R^n)\times \R^k \ni (u,\lambda) \mapsto G u+\rho \lambda - b \in \R^k
	\]
	is continuously Fr\'echet differentiable. 
	Exploiting the chain rule~\cite[Lemma~8.15]{IK08} and the Newton differentiability of 
	$\R^k \ni w \mapsto \max(0,w)\in\R^k$~\cite[Lemma~3.1]{HIK02}, 
	we have that the second component~$F_2$ is Newton differentiable. 
	
	Furthermore, according to~\cite[Prop.~4.1(ii)]{HIK02}, the mapping $v \mapsto \max(0,v)$ 
	is Newton differentiable from $L^s(0,T;\R^n)$ to $ L^r(0,T;\R^n)$ for $1\leq r < s \leq \infty$. 
	We obtain the required norm gap with $s = q$ and $r = 2$ by utilizing the smoothing properties of the 
	PDE solution operators $\Sigma$ and $\Sigma^*$, respectively. 
	For all $\theta$ satisfying $0 < \theta - 1/2 < 1/q$, there holds 
	\[
	W(0,T) \embed L^q(0,T; (H^{-1}(\Omega), H^1_0(\Omega))_{\theta, 1}),
	\]
	where $(H^{-1}(\Omega), H^1_0(\Omega))_{\theta, 1}$ denotes the real interpolation space, 
	see \eg~\cite[Sect.~1]{AMA05}. For the latter,~\cite[Chap.~4.6.1]{TRI78} yields
	\[
	(H^{-1}(\Omega), H^1_0(\Omega))_{\theta, 1} 
	\embed [H^{-1}(\Omega), H^1_0(\Omega)]_{\theta}
	= H_0^{2\theta - 1}(\Omega).
	\]
	Consequently, if we now choose $\theta = 1/2(s+1)$ (which is feasible due to our assumptions on~$s$), then 
	$\Sigma$ and $\Sigma^*$ map $L^2(0,T;H^{-1}(\Omega))$ linearly and continuously into 
	$L^q(0,T; H_0^s(\Omega))$.    
	
	According to Assumption~\ref{assu:newtondiff}, $\Psi\colon v \mapsto \sum_{j=1}^n v_j \psi_j$ maps 
	$L^{q'}(0,T;\R^n)$ linearly and continuously to $L^{q'}(0,T;H^s_0(\Omega)^*)$. Thus, 
	the Radon-Nikod\'ym property of $H^s_0(\Omega)$ implies    
	\[
	\Psi^* : L^q(0,T;H^s_0(\Omega)) = \big(L^{q'}(0,T;H^s_0(\Omega)^*)\big)^*
	\to L^q(0,T;\R^n), 
	\]
	and therefore
	\[
	L^2(0,T;\R^n) \ni u \mapsto \Psi^* \Sigma^*(\Sigma\Psi u) + \zeta - y_\textup{d}) \in 
	L^q(0,T; \R^n)
	\] 
	is affine and continuous and hence continuously Fr\'echet differentiable.  
	Moreover, if we identify $\Phi^\ell_i \in L^{q'}(0,T;\R^n)^*$, $i=1, \ldots, M_\ell$,  for a projection $\Pi_\ell$ occurring in~\eqref{eq:PCk} with its Riesz representative, denoted by the same 
	symbol, then its adjoint operator~$\Pi_\ell^*$ is given by $\R^{M_\ell} \ni v \mapsto \sum_{i=1}^{M_\ell} v_i \Phi^\ell_i \in  L^{q'}(0,T;\R^n)^*$, such that $G^* \lambda$ is given as
	\[
	G^* \lambda=\sum_{\ell=1}^k \sum_{i=1}^{M_\ell} \lambda_\ell\,a_i^\ell\Phi_i^\ell \in  L^{q'}(0,T;\R^n)^* \cong L^{q}(0,T;\R^n)
	\] and 
	\[
	\R^k \ni \lambda \mapsto G^* \lambda \in L^q(0,T;\R^n)
	\] is linear and continuous, too. 
	Hence, owing to the Newton differentiability of $\max$ and the chain rule,     
	$F_1$ is also Newton differentiable. 
\end{proof}
Now, as $F$ is Newton differentiable, we choose 
\begin{equation}\label{eq:subdiff}
H_m(\delta u, \delta \lambda) := 
\begin{pmatrix}
\chi_{\II_m}\Psi^*\Sigma^*\Sigma \Psi \delta u+ \alpha\, \delta u 
+ \chi_{\II_m}G^* \delta \lambda \\[1ex]
-\rho\,\chi_{\NN_m}\delta\lambda + \chi_{\BB_m}G\delta u
\end{pmatrix}
\end{equation}
as a generalized derivative of $F$ at a given iterate $z^m:=(u^m,\lambda^m)$ 
with the active and inactive sets for the box constraints defined (up to sets of zero Lebesgue measure) by
\[
\begin{aligned}
\AA_m^+&:=\left\{(t,j)\in(0,T)\times\{1,\ldots,n\}: 
-(\Psi^* p^m)(t)_j - (G^* \lambda)(t)_j-\tfrac{\alpha}{2} > 0\right\}, \\
\AA_m^-&:=\left\{(t,j)\in(0,T)\times\{1,\ldots,n\}: 
-(\Psi^* p^m)(t)_j -   (G^* \lambda)(t)_j + \tfrac{\alpha}{2} < 0\right\}, \\
\II_m&:=(0,T)\times\{1,\ldots,n\}\setminus\{\AA_m^+\cup \AA_m^-\},
\end{aligned}
\] 
where $p^m:=\Sigma^*(\Sigma\Psi u^m + \zeta -y_\textup{d})$, and the active and inactive cutting planes 
\[
\begin{aligned}
\BB_m&:=\{i\in\{1,\ldots,k\}:(Gu^m)_{\,i}+\rho \lambda_i^m > b_i\}, \\ 
\NN_m &:=\{1,\ldots,k\}\setminus\BB_m. 
\end{aligned}
\]
Moreover, by $\chi_{\II_m}, \chi_{\AA_m^\pm} : L^2(0,T;\R^n) \to L^2(0,T;\R^n)$ and 
$\chi_{\NN_m}, \chi_{\BB_m} : \R^k \to \R^k$, we denote the respective characteristic functions.

To compute the next iterate, we solve the following semi-smooth Newton equation 
\begin{equation}\label{eq:newton}
H_m(z^{m+1}-z^m) = -F(z^m).
\end{equation}

For the sake of simplicity, we omit the index $m$ at the inactive and active sets in the following. 
By definition of the active sets, the restriction of the first block in~\eqref{eq:newton} to $\AA^+$ and $\AA^-$,
respectively, yields 
\[
u^{m+1}=1 \quad \text{a.e.~in } \AA^+
\quad \text{and} \quad 
u^{m+1}=0\quad \text{a.e.~in } \AA^-
\] 
and the second block of~\eqref{eq:newton} restricted to $\NN$ implies $\lambda_{\ | \cal N}^{m+1}=0$. 
Therefore, we can restrict the semi-smooth Newton equation~\eqref{eq:newton} to the active components 
$\lambda^{m+1}_{\ | \BB}$ and the inactive part of the optimal control 
$u^{m+1}_{\ | \II} \in L^2(\II;\R^n)$ only, which leads to
\begin{equation}\label{eq:rnewton}
\begin{aligned}
& (\alpha I + \Psi^*\Sigma^*\Sigma\Psi\, \chi_\II^*)u^{m+1}_{\ | \II} 
+ G^* \chi_\BB^*\lambda^{m+1}_{\ | \BB}\\[-1ex]
&\qquad\qquad\qquad = \Psi^*\Sigma^*\big(y_\textup{d}  - \Sigma \Psi \chi_{\AA^+}^*\,u^{m+1}_{\ |\AA^+}
- \zeta\big) + \frac{\alpha}{2}
\quad\text{a.e.~in } \II
\end{aligned}
\end{equation}
and 
\begin{equation}\label{eq:rnewton2}
\big(G\chi_\II^* u^{m+1}_{\ | \II} \big)_\BB
= b_{\BB}-\big(G\chi_{\AA^+}^*\, u^{m+1}_{\ |\AA^+}\big)_\BB\;.
\end{equation}
Note that $\chi_\II^*$ and $\chi_{\AA^+}^*$
are the extension-by-zero operators mapping from $L^2(\II;\R^n)$ and 
$L^2(\AA^+;\R^n)$, respectively, to $L^2(0,T;\R^n)$, while
$(Gu)_\BB$ denotes the restriction to indices in~$\BB$. 
The semi-smooth Newton algorithm is now given as follows. 
\begin{algorithm}[H]
	\caption{Semi-smooth Newton method for~\eqref{eq:PCk}}\label{alg:snewton}
	\begin{algorithmic}[1]
		\STATE Choose $(u^0,\lambda^0)\in L^2(0,T;\R^n)\times\R^k$, set $
		\AA^+=\AA^-=\BB=\emptyset$ and $m=0$.
		\STATE{\label{it:nstep2} Update the active and inactive sets $\II_{m},\, 
			\AA_{m}^+,\, \AA_{m}^-,\, \BB_m$ and $\NN_m$.}
		\IF{ $\AA_{m}^+=\AA^+ \ \land \ \AA_{m}^-= \AA^-\ \land \ \BB_{m}=\BB\ \land \ m > 0$}
		\RETURN $(u^m,\lambda^m)$.
		\ELSE
		\STATE Compute $(u^{m+1},\lambda^{m+1})$ by solving the linear system~\eqref{eq:rnewton} and~\eqref{eq:rnewton2}. \label{it:nstep6}
		\STATE Set $\AA^+=\AA_{m}^+$,  $\AA^-=\AA_{m}^-$, $\BB=\BB_{m}$ and  $m=m+1$, return to \ref{it:nstep2}.
		\ENDIF
	\end{algorithmic}
\end{algorithm}

It is well known (see, \eg~\cite[Chap.~8]{IK08}) that the algorithm converges locally superlinearly if all generalized derivatives appearing in the iteration are
continuously invertible and their inverses admit a common uniform
bound. In our
case however, it is very likely that $G$ becomes rank deficient if
the number $k$ of cutting planes is large, such that the system
\eqref{eq:rnewton}--\eqref{eq:rnewton2} is no longer uniquely
solvable. In our numerical experiments, however, a moderate number of
cutting planes always sufficed and the semi-smooth Newton equation in
Step~\ref{it:nstep6} of the algorithm always admitted a unique solution for sufficiently large $\alpha$. In the case that $\alpha>0$ is small, one can only expect local superlinear convergence of the algorithm and no longer global convergence, this was also observed in our numerical experiments and a globalization would be needed for such instances. 

After each iteration of the outer approximation algorithm presented in the previous section, one has to solve a parabolic control problem~\eqref{eq:PCk} with an additional cutting plane by Algorithm~\ref{alg:snewton}. Due to this iterative structure, it is crucial to speed up the algorithm by reoptimization. More precisely, we exploit the solution of the prior outer approximation iteration to initialize the active and inactive sets in Algorithm~\ref{alg:snewton}.

The value of the Tikhonov parameter is crucial for the performance of numerical methods for the solution of optimal control problems, as already indicated above. This concerns discretization error estimates as well as convergence of optimization algorithms and conditioning of linear systems of equations arising in the latter. In case of~\eqref{eq:optprob} however, the choice of~$\alpha$ has no impact on the set of minimizers, as $u \in \{0,1\}^n \text{ a.e.~in } (0,T)$ and hence the Tikhonov term is constant. However, the convex relaxations of~\eqref{eq:optprob} considered in this paper as well their optimal values are
influenced by~$\alpha$. Thus, in order to improve the performance of Algorithm~\ref{alg:snewton}, a large value of $\alpha$ is favorable, but we expect that the quality of the dual bounds in a branch-and-bound framework will become worse for larger values of~$\alpha$. A detailed investigation of this interplay is subject to future research.

\section{Performance of the algorithm} \label{sec: casestudy}

We test the potential of our approach presented in the previous
sections by an experimental study. For this, we concentrate on the
case of a single switch with an upper bound~$\smax$ on the number of
switchings, \ie we consider
\[
D := \big\{ u \in BV(0,T)\colon \; u(t) \in \{0,1\} \text{ f.a.a.\ } t \in (0,T),\; |u|_{BV(0,T)} \leq \smax \big\}.
\]
However, we assume that $u$ is fixed to zero before the time horizon,
so that we count it as a shift if $u$ is $1$ at the beginning. The most
violated inequality for a $\Pi(u)\notin C_{D,\Pi}$, needed in
Step~\ref{it:step7} of Algorithm~\ref{alg:PC}, can then be computed in time $O(M)$~\cite{BH23}. This is fast enough to allow choosing as
intervals~$I_1,\dots,I_M$ for the projection exactly the ones given by
the discretization in time. In particular, we do not need to refine
the intervals in the course of the outer approximation algorithm.
For given~$
w\in C_{D,\Pi}$, we thus compute the most violated inequality of the form
\begin{equation*}\label{eq:alt}
\sum_{j=1}^m (-1)^{j+1} w_{i_j} \leq \Big\lfloor \frac{\smax}{2}
\Big \rfloor\;,
\end{equation*} 
where~$i_1,\dots,i_m\in \{2,\dots,M\}$ is an increasing sequence of
indices with $m-\smax$ odd and~$m>\smax$,
by choosing~$\{i_1,i_3,\dots\}$ as the local maxima
of~$w$ and~$\{i_2,i_4,\dots\}$ as the local minima of~$w$ (excluding~$1$).

The outer approximation algorithm devised in Section~\ref{sec: outerapprox} is implemented in C++, using the \textsc{DUNE}-library~\cite{SAN21} for the discretization of the PDE. The source code can be downloaded at~\url{https://github.com/agruetering/dune-MIOCP}. 
The spatial discretization uses a standard Galerkin method with continuous and piecewise linear functionals. For the state $y$ and the desired temperature $y_\textup{d}$ we also use continuous and piecewise linear functionals in time, while the temporal discretization for the controls chooses piecewise constant functionals. The resulting linear,  symmetric systems~\eqref{eq:rnewton} and~\eqref{eq:rnewton2} in each semi-smooth Newton iteration are solved by the minimum residual solver \textsc{Min-Res}~\cite{GREEN97} equipped with a suitable scalar product, induced by the temporal mass matrix, reflecting the norm of $L^2(0,T,\R^n)$ and the Euclidean scalar product in~$\R^k$,  and preconditioned with $$P=\begin{pmatrix} \alpha I & 0 \\ 0 &\tfrac{1}{\alpha} GG^\star  \end{pmatrix}\;.$$  Hereby, we approximate the spatial integrals in the weak formulation of the state and adjoint equation, respectively, by applying a Gauss-Legrendre rule with order~$3$.
The discrete systems, arising by the discretization of the state and adjoint equation, are solved by a sequential conjugate gradient solver preconditioned with AMG smoothed by SSOR.

We consider exemplary the square domain $\Omega=[0,1]^2$, the end time
$T=2$ and the form function~$\psi(x)=1.5-2(x_1-0.5)^2-2(x_2-0.5)^2.$
Moreover, in order to produce challenging instances, we generate a control~$u_\textup{d}\colon [0,T] \to [0,1]$
with a total variation~$|u_\textup{d}|_{BV(0,T)}\gg\smax$ and choose the desired
state~$y_\textup{d}$ in such a way that $u_\textup{d}$ is the optimal
solution of our relaxation as long as no cutting planes are
added. More specifically, we randomly choose $\sigma=11$ jump points~$0<
t_1< t_2 < \cdots <t_\sigma < T$ on the time grid. Then, we choose
$u_\textup{d}\colon[0,T]\to [0,1]$ as cubic spline on
$[t_{i-1},t_{i}]$, for~$1\leq i \leq \sigma+1$, where~$t_0:=0$ and
$t_{\sigma+1}:=T$, with $u_\textup{d}(t_0)=0$ and
$u_\textup{d}(t_{\sigma+1})=0.5$.  The latter condition
guarantees~$p^\star(T)=0$ for the adjoint state
\[p^\star(t,x)=-\alpha c (u_\textup{d}(t)-\tfrac{1}{2})\sin(\pi
x_1)\sin(\pi x_2),\] where $c$ is the inverse of the value
$\int_\Omega \psi(x)\sin(\pi x_1)\sin(\pi x_2) \, \d x$ and $\alpha$ is
the Tikhonov parameter. By setting
\[y_\textup{d}(t,x):=S(u_\textup{d})+\partial_t p^\star(t,x) + \Delta
p^\star(t,x),\] the optimal solution of our relaxation without cutting
planes is given as $u^\star=u_\textup{d}$ and~$p^\star$ represents the
optimal adjoint state. In the generation of the instance, we compute~$S(u_\textup{d})$ on a time
grid with $N_t=400$ time intervals, whereas
the outer approximation is performed on a coarser grid.

In all experiments, we use a uniform spatial
triangulation of~$\Omega$ with $30\times 30$ nodes, while
experimenting with different temporal resolutions.
The Tikhonov parameter was always set to~$\alpha=10^{-2}$. For the update of active cutting planes we
chose~$\rho=10^{-5}$; see Section~\ref{sec: numsol}. The cutting plane
algorithm stops as soon as the violation of the most violated cutting plane falls below~$1\%$ of the right hand side, the
control is considered feasible for~\eqref{eq:PC} in this case. Note that the validity of the
lower bound is not compromised by this.

All computations have been performed on a 64bit Linux system with an
Intel Xeon E5-2640 CPU @ 2.5 GHz and $32$ GB RAM.

We first illustrate the development of lower bounds over time; see
Figure~\ref{fig:time}. Here, we used a typical instance with~$\smax=2$
and a time grid with~$N_t=100$ intervals. Each cross
corresponds to the lower bound (y-axis) obtained after adding another
cutting plane, where the x-axis represents the time needed (in
CPU hours) to obtain this bound. It can be seen that the bounds improve
very quickly in the first cutting plane iterations and then continue
to increase slowly. When using the lower bounds within a
branch-and-bound scheme, this suggests to generate only few cutting
planes before resorting to branching. For comparison, we also show the
development of lower bounds in case no reoptimization is used; this is
marked by circles. It can be observed that reoptimization significantly
decreases running times.

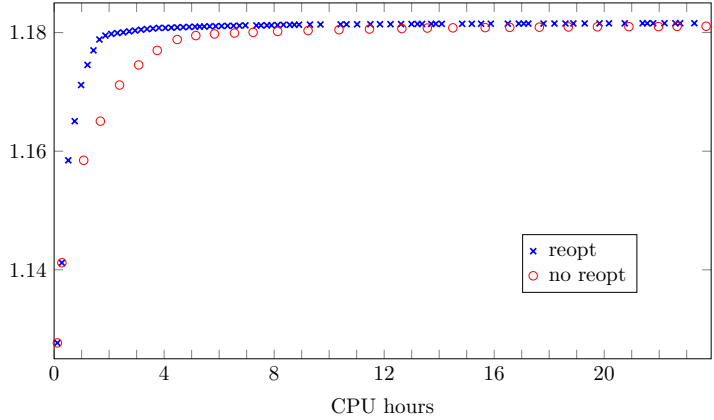
\begin{figure}[htb]
	\begin{center}
		\begin{tikzpicture}[scale=0.8]
		\begin{axis}[
		width=12.5cm,
		height=7.5cm,
		xmin=-0.2,
		xmax=86000,
		ymin=0.01125,
		ymax=0.01185,
		xtick={0, 3600, 7200,10800,14400,18000, 21600, 25200,28800,32400,36000,39600, 43200,46800, 50400, 54000, 57600, 61200, 64800, 68400, 72000, 75600, 79200,82800, 86400},
		xticklabels={0,$ $, $ $,$ $, 4, $ $, $ $, $ $, 8, $ $, $ $, $ $, 12, $ $, $ $, $ $, 16, $ $, $ $, $ $, 20, $ $, $ $, $ $, 24},
		xlabel={CPU hours},
		xtick scale label code/.code={},
		ytick scale label code/.code={},
		every y tick label/.append style={ /pgf/number format/.cd,precision=2,fixed zerofill}, 
		legend style={at={(0.8,0.35)},anchor=north},
		legend cell align={left}
		]
		\addplot[only marks, mark=x, thick, color=blue] coordinates {
			(419.01,0.01127681228)
			(1017.69,0.0114118665)
			(1856.04,0.01158463402)
			(2695.39,0.01165057301)
			(3533.95,0.01171155968)
			(4371.83,0.01174543186)
			(5151.12,0.01176999434)
			(5929.65,0.01178831773)
			(6709.03,0.01179491041)
			(7487.15,0.0117976612)
			(8266.58,0.01179920585)
			(9044.79,0.01180031455)
			(9823.95,0.01180205365)
			(10603.18,0.01180337827)
			(11381.87,0.01180481165)
			(12160.39,0.01180594658)
			(12939.38,0.01180680686)
			(13717.98,0.01180791039)
			(14556.02,0.01180783353)
			(15335.28,0.01180826737)
			(16115.14,0.01180874175)
			(16893.8,0.0118091338)
			(17673.18,0.01180956081)
			(18453.1,0.0118097421)
			(19051.56,0.01180998336)
			(19650.29,0.01181022889)
			(20369.53,0.0118105064)
			(21147.31,0.01181074958)
			(21926.36,0.01181102589)
			(22704.65,0.01181130672)
			(23483.08,0.01181140672)
			(24261.31,0.01181160064)
			(25039.97,0.01181216667)
			(26478.24,0.01181198333)
			(27077.03,0.01181212127)
			(27855.73,0.01181221504)
			(28454.73,0.01181235245)
			(29233.94,0.0118126086)
			(30012.94,0.01181309437)
			(30852.39,0.01181304717)
			(31450.47,0.01181314419)
			(32049.52,0.01181327658)
			(33488.9,0.01181344812)
			(34867.2,0.01181360361)
			(37500.87,0.01181376461)
			(38279.67,0.01181393302)
			(39657.61,0.01181396363)
			(41394.86,0.01181402938)
			(42652.76,0.0118140814)
			(44031.3,0.01181414002)
			(45408.8,0.01181426897)
			(46784.77,0.0118144051)
			(47563.35,0.01181445085)
			(48162.18,0.01181453019)
			(49361.76,0.01181458085)
			(49960.92,0.01181465704)
			(50742.28,0.01181466944)
			(53411.08,0.0118147182)
			(54676.44,0.01181478046)
			(55887.55,0.01181484058)
			(57156.28,0.01181490844)
			(59382.89,0.01181500305)
			(60821.99,0.01181508871)
			(61480.37,0.0118151318)
			(62078.93,0.01181518931)
			(64054.42,0.01181517619)
			(65491.3,0.01181524467)
			(66927.95,0.01181530431)
			(67962.38,0.01181535383)
			(69448.29,0.01181542794)
			(71385.26,0.01181549083)
			(72622.59,0.01181555477)
			(74689.28,0.01181561373)
			(77030.48,0.01181563146)
			(77630.91,0.01181565788)
			(78411.48,0.01181568856)
			(79894.4,0.01181572232)
			(81318.53,0.01181573849)
			(81938.61,0.01181575885)
			(83797.74,0.01181578605)};
		\addlegendentry{~reopt};
		\addplot[only marks, mark=o, color=red] coordinates {
			(420.9,0.01127681228)
			(1015.53,0.0114118665)
			(3871.15,0.01158463402)
			(6053.11,0.01165057301)
			(8565.2,0.01171155968)
			(11065.13,0.01174543186)
			(13498.52,0.01176999433)
			(16107.62,0.01178831773)
			(18548.45,0.0117949104)
			(20986.2,0.01179766118)
			(23605.49,0.01179920585)
			(26046.1,0.01180031453)
			(29258.9,0.01180205366)
			(33241.91,0.01180337827)
			(37285.9,0.01180481166)
			(41279.49,0.01180594658)
			(45512.73,0.01180680686)
			(48849.42,0.01180760105)
			(52188.89,0.0118079589)
			(56422.1,0.01180833598)
			(59640.19,0.01180873772)
			(63510.86,0.01180910777)
			(67310.86,0.01180940777)
			(71110.69,0.01180968637)
			(75232.34,0.01180992826)
			(79103.19,0.01181008954)
			(81538.67,0.01181036491)
			(85336.05,0.01181071575)};
		\addlegendentry{~no reopt};
		\end{axis}
		\end{tikzpicture}
		\caption{Temporal development of bounds.}\label{fig:time}
	\end{center}
\end{figure}

The Tikhonov term has an impact on the performance of Algorithm~\ref{alg:PC}, as well as on the quality of the bounds. The larger $ \alpha$ is, the worse the bounds become, but the faster the convex relaxations can be solved. This can also be observed in Figure~\ref{fig:alpha}, where we solved the same instance of Figure~\ref{fig:time}, \ie we used the same desired state~$y_\textup{d}$, with different Tikhonov parameters. The results show that choosing a small $\alpha$ is generally favorable, but for very small $\alpha$ there occurs a trade-off between the quality  of the dual bounds and the convergence rate. Within a branch-and-bound framework, one needs to empirically investigate whether a good quality or a quick computation of the dual bounds for small $\alpha$ have a greater influence on the overall performance.
\begin{figure}[htb]
	\centering
	\begin{tikzpicture}[scale=0.8]
	\begin{axis}[
	width=12.5cm,
	height=9cm,
	xmin=-0.2,
	xmax=86000,
	ymin=0.0112,
	ymax=0.0131,
	xtick={0, 3600, 7200,10800,14400,18000, 21600, 25200,28800,32400,36000,39600, 43200,46800, 50400, 54000, 57600, 61200, 64800, 68400, 72000, 75600, 79200,82800, 86400},
	xticklabels={0,$ $, $ $,$ $, 4, $ $, $ $, $ $, 8, $ $, $ $, $ $, 12, $ $, $ $, $ $, 16, $ $, $ $, $ $, 20, $ $, $ $, $ $, 24},
	xlabel={CPU hours},
	xtick scale label code/.code={},
	ytick scale label code/.code={},
	every y tick label/.append style={ /pgf/number format/.cd,precision=2, fixed zerofill}, 
	legend style={at={(0.85,0.7)},anchor=north},
	legend cell align={left}
	]
	\addplot[only marks, mark=asterisk, cyan!50!black, mark size=2pt,thick] coordinates {
		(1379.42,0.011983257184)
		(4218.01,0.01209474931)
		(8216.51,0.0121953626)
		(11422.53,0.01229167796)
		(14569.13,0.01234396277)
		(18015.1,0.01239531646)
		(21584.15,0.0124399101)
		(25334.57,0.01247386242)
		(29933.74,0.01252030212)
		(34289.14,0.01256261253)
		(37741.48,0.01259395046)
		(41434.81,0.01262352212)
		(45372.64,0.01265621933)
		(48803.13,0.01268001014)
		(56097.2,0.01271602994)
		(59525.78,0.01274315026)
		(64102.43,0.01277803161)
		(70515.44,0.01280572117)
		(77133.86,0.01283407706)
		(80967.26,0.01286479576)
		(85580.23,0.01291454032)};
	\addlegendentry{~$\alpha=0.002$};
	\addplot[only marks, mark=triangle*, black, mark size=2pt] coordinates {
		(1331.64,0.0119471935)
		(3762.94,0.01208586696)
		(5769.92,0.01224046891)
		(8835.25,0.01235346369)
		(11983.33,0.01240912213)
		(14776.4,0.01244120975)
		(18437.49,0.01249026204)
		(21599.02,0.01251965443)
		(24819.89,0.01255531812)
		(27793.38,0.01258958732)
		(29974.71,0.01261474418)
		(32216.56,0.01263541718)
		(35425.2,0.01267197606)
		(37669.14,0.01269271529)
		(41068.6,0.01271614425)
		(45934.06,0.01275048464)
		(49386.44,0.0127715442)
		(52969.4,0.01280128057)
		(56603.8,0.01285342078)
		(59090.27,0.01287563025)
		(62383.54,0.01289924122)
		(65356.79,0.01291425369)
		(68465.16,0.01292788749)
		(71887.7,0.01294075016)
		(75009.78,0.01294993183)
		(77379.01,0.01295951629)
		(79671.36,0.01296697376)
		(81767.64,0.01297667166)
		(83989.94,0.01298444932)
		(87196.63,0.01299222541)}; 
	\addlegendentry{~$\alpha=0.003$};
	\addplot[only marks, mark=o, red, mark size=2pt] coordinates {
		(1172.87,0.01183193503)
		(3277.42,0.01199846151)
		(5196.89,0.01218300478)
		(7847.02,0.01228193816)
		(10565.96,0.01236339247)
		(13657.19,0.01241696139)
		(15731.84,0.01246407393)
		(17894.67,0.01250500287)
		(20053.57,0.01252881036)
		(22164.51,0.01254919467)
		(24975.84,0.01257204768)
		(27765.71,0.01258935575)
		(30662.67,0.0126096747)
		(32704.79,0.01262804841)
		(34746.89,0.01264243697)
		(36895.89,0.01265704602)
		(38926.92,0.01266651921)
		(42020.19,0.01267533298)
		(43071,0.0126807832)
		(44105.81,0.01268777461)
		(46064.95,0.01269582944)
		(47107.64,0.012701732)
		(48160.37,0.01270770473)
		(51172.8,0.01271271684)
		(52202.8,0.01271705257)
		(54268.16,0.01272037406)
		(55238.77,0.01272359152)
		(56209.68,0.01272905165)
		(57968.42,0.01273331865)
		(59545.44,0.01273824895)
		(60518.9,0.01274191275)
		(62944.86,0.0127464956)
		(63915.08,0.01275094197)
		(64947.12,0.01275312368)
		(65918.12,0.01275623446)
		(66887.55,0.01275866189)
		(68698.82,0.01276051908)
		(69657.74,0.01276398123)
		(71456.86,0.01276732977)
		(72416.48,0.01276976975)
		(74215.03,0.01277207988)
		(75174.76,0.01277489797)
		(76794.42,0.01277696726)
		(77752.95,0.0127792644)
		(79310.31,0.01278145172)
		(82247.85,0.01278446185)
		(84045.71,0.01278673449)
		(85004,0.01279036863)
		(86800.63,0.01279256298)};
	\addlegendentry{~$\alpha=0.005$};
	\addplot[only marks, mark=x, blue,mark size=2pt,thick] coordinates {
		(419.01,0.01127681228)
		(1017.69,0.0114118665)
		(1856.04,0.01158463402)
		(2695.39,0.01165057301)
		(3533.95,0.01171155968)
		(4371.83,0.01174543186)
		(5151.12,0.01176999434)
		(5929.65,0.01178831773)
		(6709.03,0.01179491041)
		(7487.15,0.0117976612)
		(8266.58,0.01179920585)
		(9044.79,0.01180031455)
		(9823.95,0.01180205365)
		(10603.18,0.01180337827)
		(11381.87,0.01180481165)
		(12160.39,0.01180594658)
		(12939.38,0.01180680686)
		(13717.98,0.01180791039)
		(14556.02,0.01180783353)
		(15335.28,0.01180826737)
		(16115.14,0.01180874175)
		(16893.8,0.0118091338)
		(17673.18,0.01180956081)
		(18453.1,0.0118097421)
		(19051.56,0.01180998336)
		(19650.29,0.01181022889)
		(20369.53,0.0118105064)
		(21147.31,0.01181074958)
		(21926.36,0.01181102589)
		(22704.65,0.01181130672)
		(23483.08,0.01181140672)
		(24261.31,0.01181160064)
		(25039.97,0.01181216667)
		(26478.24,0.01181198333)
		(27077.03,0.01181212127)
		(27855.73,0.01181221504)
		(28454.73,0.01181235245)
		(29233.94,0.0118126086)
		(30012.94,0.01181309437)
		(30852.39,0.01181304717)
		(31450.47,0.01181314419)
		(32049.52,0.01181327658)
		(33488.9,0.01181344812)
		(34867.2,0.01181360361)
		(37500.87,0.01181376461)
		(38279.67,0.01181393302)
		(39657.61,0.01181396363)
		(41394.86,0.01181402938)
		(42652.76,0.0118140814)
		(44031.3,0.01181414002)
		(45408.8,0.01181426897)
		(46784.77,0.0118144051)
		(47563.35,0.01181445085)
		(48162.18,0.01181453019)
		(49361.76,0.01181458085)
		(49960.92,0.01181465704)
		(50742.28,0.01181466944)
		(53411.08,0.0118147182)
		(54676.44,0.01181478046)
		(55887.55,0.01181484058)
		(57156.28,0.01181490844)
		(59382.89,0.01181500305)
		(60821.99,0.01181508871)
		(61480.37,0.0118151318)
		(62078.93,0.01181518931)
		(64054.42,0.01181517619)
		(65491.3,0.01181524467)
		(66927.95,0.01181530431)
		(67962.38,0.01181535383)
		(69448.29,0.01181542794)
		(71385.26,0.01181549083)
		(72622.59,0.01181555477)
		(74689.28,0.01181561373)
		(77030.48,0.01181563146)
		(77630.91,0.01181565788)
		(78411.48,0.01181568856)
		(79894.4,0.01181572232)
		(81318.53,0.01181573849)
		(81938.61,0.01181575885)
		(83797.74,0.01181578605)
		(86768.82,0.01181581146)};
	\addlegendentry{~$\alpha=0.01$};		
	\end{axis}
	\end{tikzpicture}
	\caption{Temporal development of bounds for different $\alpha$.}\label{fig:alpha}
\end{figure}
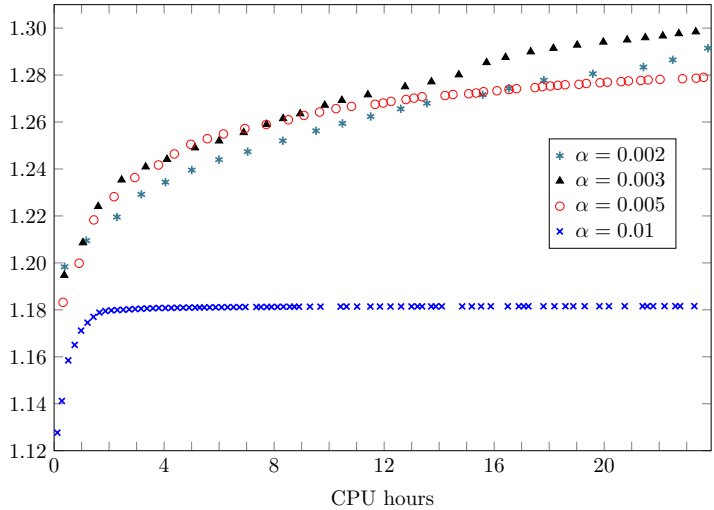

We next show the typical behavior of the optimal solutions of the
relaxation when adding more and more cutting planes. For the example
shown in Figure~\ref{fig:optsol}, we again have~$N_t=100$
and~$\smax=2$. Before adding the first cutting plane, the total
variation is not bounded by any constraint; we have~$|u^0|_{BV(0,T)}=8.74$ then. Adding
cutting planes quickly changes the shape of the optimal
solutions~$u^i$ as well as their total variation, which however does
not necessarily decrease monotonously. We emphasize that
neither the shape of~$u^i$ nor its total variation is directly
relevant for our approach, since we only aim at computing as tight lower
bounds as possible.
\begin{figure}[htb]
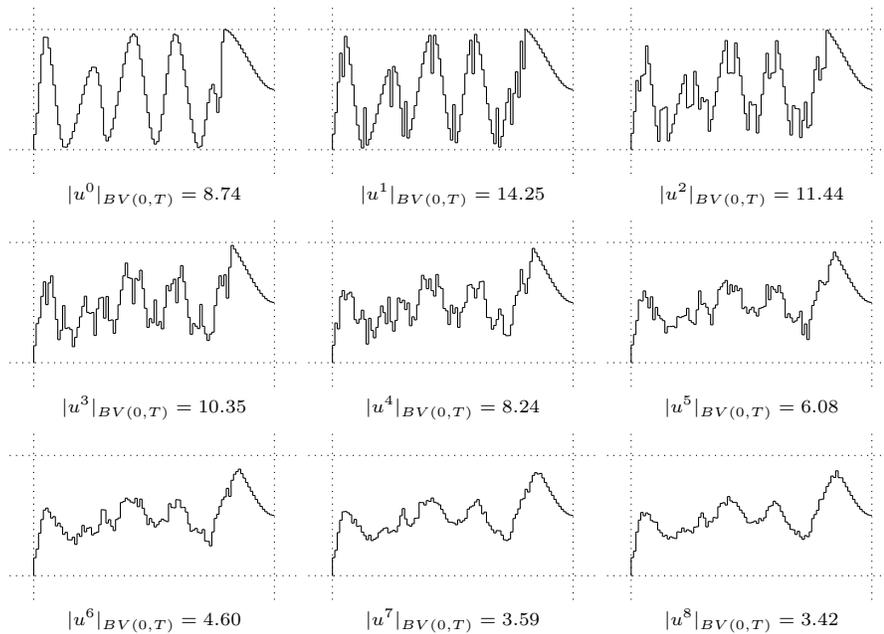

	\begin{center}
		\bigskip\bigskip
		\tikzset{every picture/.style={scale=1.6}}
		
		\caption{Development of optimal solutions.}\label{fig:optsol}
	\end{center}
\end{figure}

Finally, we investigate the impact of the number of time intervals
chosen for the discretization. Figure~\ref{fig:Nt}
demonstrates the temporal development of lower bounds for different
numbers $N_t$ and $\smax=2$. For a better comparison, we recalculate
the resulting lower bounds (y-axis) with a finer temporal
discretization, namely~$N_t=400$; note that this may lead to non-monotonous
bounds. We observe that a coarser time grid quickly leads to better
bounds, however, the accuracy of the lower bounds suffers
enormously. In fact, the bounds obtained for a given discretization
may not remain valid for a finer temporal grid.

\begin{figure}[htb]
	\centering
	\begin{tikzpicture}[scale=0.8]
	\begin{axis}[
	width=12.5cm,
	height=7.5cm,
	xmin=-0.2,
	xmax=86400,
	ymin=0.01065,
	ymax=0.0117,
	xtick={0, 3600, 7200,10800,14400,18000, 21600, 25200,28800,32400,36000,39600, 43200,46800, 50400, 54000, 57600, 61200, 64800, 68400, 72000, 75600, 79200,82800, 86400 },
	xticklabels={0,$ $, $ $,$ $, 4, $ $, $ $, $ $, 8, $ $, $ $, $ $, 12, $ $, $ $, $ $, 16, $ $, $ $, $ $, 20,$ $,$ $,$ $, 24},
	xlabel={CPU hours},
	xtick scale label code/.code={},
	ytick scale label code/.code={},
	every y tick label/.append style={ /pgf/number format/.cd,precision=2, fixed zerofill}, 
	legend style={at={(0.8,0.35)},anchor=north},
	legend cell align={left}
	]
	\addplot[only marks, mark=x, color=blue, mark size=1.5pt] coordinates {
		(108.56,0.01131019183)
		(263.53,0.01148970867)
		(480.6,0.01160036891)
		(667.13,0.0116574099)
		(866.36,0.01162155993)
		(1063.03,0.01163834517)
		(1259.63,0.01163593139)
		(1456.27,0.01161978442)
		(1652.87,0.01161763453)
		(1849.52,0.0116287118)
		(2015.9,0.01162901939)
		(2212.54,0.0116292398)
		(2409.14,0.01162487488)
		(3226.03,0.0116251867)
		(3427.63,0.0116244224)
		(3582.95,0.01162490327)};
	\addlegendentry{~$N_t=25$};
	\addplot[only marks, mark=o, color=red, mark size=1.5pt] coordinates {
		(209.22,0.01080080574)
		(508.36,0.01113409542)
		(927.13,0.01141409213)
		(1345.86,0.01146862651)
		(1734.69,0.01148391876)
		(2123.51,0.01148716228)
		(2512.2,0.01148926734)
		(2901.17,0.01149296179)
		(3290.16,0.01149351255)
		(3679.02,0.01149570716)
		(4067.83,0.0114985088)
		(4456.53,0.01150320171)
		(4845.38,0.01150560679)
		(5234.31,0.0115136072)
		(5653.26,0.01151344632)
		(6042,0.01151672484)
		(6430.73,0.01151723537)
		(7148.8,0.01151839416)
		(7537.68,0.01151849243)
		(7926.39,0.01151990539)
		(8315.07,0.01152116855)
		(8703.76,0.01152185164)
		(9092.5,0.01152618635)
		(9869.86,0.01152975045)
		(10258.5,0.01152898234)
		(10647.23,0.01152759928)
		(11335.63,0.0115273191)
		(11664.99,0.01152687063)
		(12053.75,0.01152751421)
		(13071.98,0.01152598081)
		(13460.73,0.01152688557)
		(14178.54,0.01152702959)
		(14866.42,0.01152760481)
		(15255.16,0.01152871078)
		(15973.88,0.01152798047)
		(17141.48,0.01152863229)
		(17877.86,0.01152846996)
		(18584.04,0.01152713859)
		(18921.55,0.01152761303)
		(19597.47,0.01152748037)
		(19999.08,0.01152715558)
		(20523.46,0.0115273914)
		(21141.51,0.01152679912)
		(22747.78,0.01152701681)
		(23395.38,0.0115270928)};
	\addlegendentry{~$N_t=50$};
	\addplot[only marks, mark=triangle*, mark size=1.5pt] coordinates {
		(436.69,0.01068722176)
		(1060.57,0.01089306183)
		(1933.47,0.01122871006)
		(2792.61,0.01133757599)
		(3644.54,0.01138968922)
		(4454.93,0.01141140455)
		(5266.19,0.0114225526)
		(6058.11,0.01142882746)
		(6845.59,0.01143391348)
		(7642.58,0.01143881487)
		(8452.31,0.01144335714)
		(9261.65,0.01144671483)
		(10052.51,0.01144934076)
		(10847.85,0.01145171391)
		(11659.27,0.01145399657)
		(12468.59,0.0114574224)
		(13279.69,0.01146141733)
		(14153.72,0.0114632694)
		(14964.21,0.01146565166)
		(15775.33,0.01146913375)
		(16646.73,0.01147056117)
		(17457.04,0.01147201401)
		(18268.81,0.01147322098)
		(19080.16,0.0114754614)
		(19891.46,0.01147779069)
		(20765.76,0.01147864187)
		(21576.67,0.01148033213)
		(22451.15,0.01148088572)
		(23263.36,0.01148252369)
		(24137.63,0.01148354241)
		(24928.93,0.01148444811)
		(25719.97,0.01148597117)
		(26512.09,0.01148813049)
		(27946.43,0.01148825672)
		(28757.64,0.01148951837)
		(30068.78,0.01148964224)
		(30880.16,0.0114908208)
		(31690.11,0.0114913646)
		(33108.28,0.01149214424)
		(33912.59,0.01149311225)
		(35409.53,0.01149327744)
		(36219.44,0.01149417131)
		(37716.36,0.01149371107)
		(38524.43,0.01149372949)
		(39310.44,0.01149377328)
		(40096.26,0.01149426276)
		(42379.69,0.01149535768)
		(43170.88,0.01149616424)
		(44645.48,0.01149697401)
		(46038.53,0.01149793406)
		(46704.66,0.01149868826)
		(47491.53,0.01149877763)
		(48763.04,0.01149918073)
		(49565.19,0.0114995429)
		(51000.69,0.01150001503)
		(51811.37,0.01150092139)
		(52497.78,0.01150120911)
		(53932.65,0.01150206771)
		(56465.87,0.01150294283)
		(57901.12,0.01150368724)
		(61389.23,0.01150427268)
		(62190.32,0.01150498694)
		(62976.51,0.01150482825)
		(64397.75,0.01150532404)
		(65887.9,0.0115058081)
		(66699.2,0.01150641602)
		(67509.47,0.01150671789)
		(68998.5,0.01150678309)
		(70214.79,0.01150706095)
		(70945.41,0.01150737469)
		(72220,0.01150740574)
		(73008.29,0.01150776479)
		(75069.53,0.01150806509)
		(76469.35,0.01150829997)
		(77931.17,0.01150835541)
		(79389.95,0.01150857913)
		(80178.24,0.01150874185)
		(81390.78,0.01150882761)
		(84117.42,0.01150911078)};
	\addlegendentry{~$N_t=100$};
	\end{axis}
	\end{tikzpicture}
	\caption{Development of lower bounds for refined time grid.}\label{fig:Nt}
\end{figure}
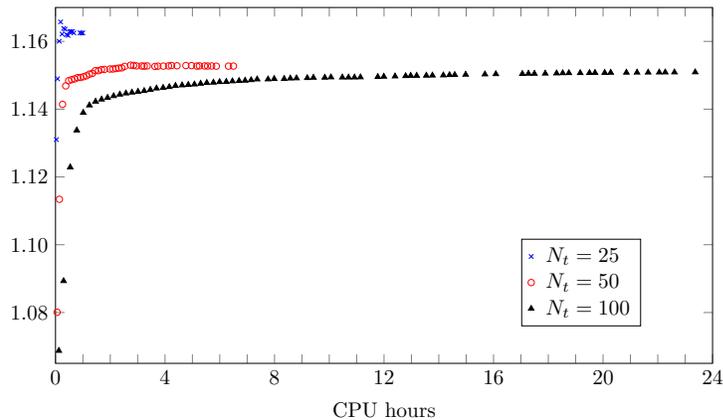
In a branch-and-bound
scheme, where larger parts of the switching structure will be fixed by the branching decisions, an adaptive
discretization of the problem may be rewarding. Such an approach could be practicable within our outer
approximation algorithm in function space, this is left as future work.

\appendix
\section{Existence of Lagrange multipliers}\label{sec: appendix}

This appendix shows the existence of Lagrange-multipliers for box constraints and 
finitely many linear inequality constraints, as appearing in the relaxation~\eqref{eq:PCk}.
A similar result under slightly less restrictive assumptions is shown in \cite{Wac22}, but, for convenience of the reader we present a proof based on standard arguments in detail.  
For that, we consider problems of the form
\begin{equation}\label{eq:p}
\left\{\quad
\begin{aligned}
\min \quad & f(u)\\
\text{s.t.} \quad & u_a(\xi) \leq u(\xi) \leq u_b(\xi) \quad \text{f.a.a.\ $\xi$ in } \Lambda\\
& G u \leq b
\end{aligned}
\right.
\end{equation}
where $\Lambda\subset \R^d$, $d \in \N$, is bounded and Lebesgue-measurable
and $f: L^2(\Lambda) \to \R$ is continuously Fr\'echet differentiable.
Moreover, $G: L^2(\Lambda) \to \R^m$, $m\in \N$ is linear and bounded and $b\in \R^m$ is given.
Finally, $u_a, u_b \in L^\infty(\Lambda)$ satisfy
\begin{equation}\label{eq:slater0}
u_a(\xi) + \delta \leq u_b(\xi) \quad \text{f.a.a.\ } \xi\in \Lambda
\end{equation}
with some $\delta > 0$. 

Note that $L^\infty(\Lambda) \embed L^2(\Lambda)$, since $\Lambda$ is bounded.
We will frequently regard $G$ and $f$ as mappings with domain $L^\infty(\Lambda)$ and, 
with a little abuse of notation, these maps are denoted by the same symbols.
Clearly, they are also Fr\'echet differentiable as mappings 
with domain in $L^\infty(\Lambda)$.

In the following, let $\bar u\in L^\infty(\Lambda)$ be a locally optimal solution of~\eqref{eq:p}. 
If we define the convex set $C := \{ u \in L^\infty(\Lambda) : G u \leq b\}$,
then~\eqref{eq:p} is equivalent to 
\[
\eqref{eq:p}
\quad \Longleftrightarrow \quad 
\left\{
\begin{aligned}
\min \quad & f(u) \\
\text{s.t.} \quad  & u - u_a \in K, \quad u_b - u \in K, \quad u \in C
\end{aligned}
\right.
\]
with $K := \{v \in L^\infty(\Lambda): v \geq 0 \text{ a.e.~in } \Lambda\}$.
Note that $K$ admits a non-empty interior as subset of $L^\infty(\Lambda)$.
Furthermore, due to the linearity and continuity of the mapping $G$ from $L^\infty(\Lambda)\embed L^2(\Lambda)$ to $\R^m$, 
the set $C$ is convex and closed.

In addition to~\eqref{eq:slater0}, we suppose that the Slater condition is fulfilled, i.e., 
we assume that there is a function $\hat u\in L^\infty(\Lambda)$ such that 
\begin{equation}\label{eq:slaterbdg}
G \hat u \leq b, \quad 
u_a(\xi) + \rho \leq \hat u(\xi) \leq u_b(\xi) - \rho
\end{equation}
with $\rho > 0$. Since Slater's condition implies Robinson's constraint qualification, 
	\cite[Theorem~3.9]{BS00} yields the existence of
	Lagrange multipliers $\mu_a, \mu_b\in L^\infty(\Lambda)^*$ such that 
\begin{gather}
\dual{f'(\bar u) + \mu_b - \mu_a}{u - \bar u}_{L^\infty(\Lambda)^*, L^\infty(\Lambda)} \geq 0
\quad \forall\, u \in C, \label{eq:gradeq}\\
\mu_b \in K^+, \quad \dual{\mu_b}{\bar u - u_b}_{L^\infty(\Lambda)^*, L^\infty(\Lambda)} = 0, 
\quad \bar u \leq u_b \text{ a.e.~in }\Lambda,\label{eq:compla}\\
\mu_a \in K^+, \quad \dual{\mu_a}{u_a - \bar u}_{L^\infty(\Lambda)^*, L^\infty(\Lambda)} = 0, 
\quad u_a \leq \bar u \text{ a.e.~in }\Lambda\;, \label{eq:complb}
\end{gather}
where the dual cone is given by
\[
K^+ := \{ \nu\in L^\infty(\Lambda)^* : \dual{\nu}{v}_{L^\infty(\Lambda)^*, L^\infty(\Lambda)} \geq 0
\quad \forall\, v\in K\}\;.
\]
In view of the definition of $C$, the gradient equation in~\eqref{eq:gradeq} is equivalent to 
\begin{equation}\label{eq:vi}
\dual{f'(\bar u) + \mu_b - \mu_a}{s}_{L^\infty(\Lambda)^*, L^\infty(\Lambda)} \geq 0
\quad \forall\, s \in G^{-1} \cone(\R^m_- - (G \bar u-b))\;,
\end{equation}
where $\cone$ denotes the conic hull and $\R^m_- := \{v\in \R^m: v \leq 0\}$.
The conic hull is given by
\[
\cone(\R^m_- - (G \bar u-b))
= \cone\big(\{ - \mathrm{e}_1, \ldots , -\mathrm{e}_m, - G\bar u + b\}\big)
\]
and, as the conic hull of finitely many points in $\R^m$, it is therefore closed.
For its polar cone we find by elementary calculus that 
\begin{equation}\label{eq:polcone}
\cone(\R^m_- - (G \bar u-b))^\circ = \cone\big(\{ \mathrm{e}_i : i\in \AA\}\big)\;,
\end{equation}        
where  $\AA := \{i \in \{1, \ldots, m\}: (G u)_i = b_i\}$ and $\mathrm{e}_i\in \R^m$, $i=1, \ldots., m$, denote the Euclidean unit vectors.
Moreover, the following holds true:

\begin{lemma}\label{lem:farkaspre}
		The set $G^* \cone(\R^m_- - (G\bar u - b))^\circ$ is a weakly-$\ast$ closed subset of $L^\infty(\Lambda)^*$.
	\end{lemma}
	
	\begin{proof}
		Because $G$ maps $L^2(\Lambda)$ linearly and continuously to $\R^m$, there exist functionals $g_i\in L^2(\Lambda)^*\cong L^2(\Lambda)$, $i = 1, \ldots, m$, such that 
		\[
		G u = \big( \dual{g_i}{u} \big)_{i=1}^m\;.
		\]
		With a slight abuse of notation, we denote the application of $g_i$ to functions in $L^\infty(\Lambda)$ 
		by the same symbol. Direct computation shows that 
		\begin{equation}\label{eq:Gadj}
		G^* : \R^m \to L^\infty(\Lambda)^*, \quad 
		G^* \lambda = \sum_{i=1}^m \lambda_i g_i
		\end{equation}
		such that~\eqref{eq:polcone} implies $G^* \cone(\R^m_- - (G\bar u-b))^\circ 	= \cone\big(\{g_i : i \in \AA\}\big)$.
		By~\cite[Prop.\ 2.41]{BS00}, the set $\cone\big(\{g_i : i \in \AA\}\big)$ is weak-* closed, so that the assertion follows.
\end{proof}

Thanks to Lemma~\ref{lem:farkaspre}, all prerequisites of the generalized Farkas lemma are fulfilled,  
see \eg~\cite[Prop.~2.4.2]{SCHI07}.
Therefore,~\eqref{eq:vi} is equivalent to the existence of a multiplicator
$\lambda \in \cone(\R^m_- - (G \bar u-b)) = \cone(\{ \mathrm{e}_i : i\in \AA\})$, cf.~\eqref{eq:polcone}, 
such that 
\begin{equation}
- f'(\bar u) - \mu_b + \mu_a - G^* \lambda = 0 
\quad \text{in } L^\infty(\Lambda)^*\;. \label{eq:gradeq2}
\end{equation} 
Now, we are in the position to prove the desired multiplier theorem:

\begin{theorem}\label{thm:lagrange}
	Suppose that a Slater point fulfilling~\eqref{eq:slaterbdg} exists and 
	let $\bar u$ be a locally optimal solution to~\eqref{eq:p}. Then there exist Lagrange multipliers 
	$\lambda \in \R^m$ and $\mu_a, \mu_b \in~L^2(\Lambda)$ such that 
	\begin{gather}
	f'(\bar u) + \mu_b - \mu_a + G^*\lambda = 0 \quad \text{a.e.~in }\Lambda, 
	\label{eq:gradeq3}\\
	\mu_a \geq 0, \quad \mu_a (\bar u - u_a) = 0, \quad \bar u \geq u_a \quad \text{a.e.~in }\Lambda, 
	\label{eq:compla2}\\
	\mu_b \geq 0, \quad \mu_b (\bar u - u_b) = 0, \quad \bar u \leq u_b \quad \text{a.e.~in }\Lambda, 
	\label{eq:complb2}\\
	\lambda \geq 0, \quad \lambda^\top (G\bar u - b) = 0, \quad G \bar u \leq b\;.
	\label{eq:compllambda}
	\end{gather}
\end{theorem}

\begin{proof}
	Let $u\in L^\infty(\Lambda)$ with $u_a \leq u \leq u_b$ a.e.~in $\Lambda$ be arbitrary. 
	Inserting $v = u - \bar u$ in~\eqref{eq:gradeq2}, we obtain 
	\[
	\begin{aligned}
	0 &= \dual{f'(\bar u) + G^* \lambda}{u - \bar u} 
	+ \dual{\mu_b}{u - u_b} + \dual{\mu_b}{u_b - \bar u}
	+ \dual{\mu_a}{u_a - u} + \dual{\mu_a}{\bar u - u_a}\\
	& \leq \dual{f'(\bar u) + G^* \lambda}{u - \bar u}\;,
	\end{aligned}
	\]
	where we used~\eqref{eq:compla} and~\eqref{eq:complb} for the last estimate.
	Since $f'(\bar u) \in L^2(\Lambda)^* \cong L^2(\Lambda)$ and $G^*$ maps $\R^m$ to $L^2(\Lambda)$ 
	by the regularity of $g_i$, $i=1, \ldots, m$, the last inequality is equivalent to
	\begin{equation}\label{eq:vi2}
	\int_\Omega ( f'(\bar u) + G^* \lambda)(u - \bar u) \, \d \xi \geq 0
	\quad \forall \,u\in L^\infty(\Lambda) : u_a \leq u \leq u_b \text{ a.e.~in } \Lambda.
	\end{equation}
	Now, we introduce the functions $\mu_a, \mu_b\in L^2(\Lambda)$ by
	\[
	\begin{aligned}
	\mu_a(\xi) := \max\{( f'(\bar u) + G^* \lambda)(\xi), 0\}, \quad 
	\mu_b(\xi) := - \min\{( f'(\bar u) + G^* \lambda)(\xi), 0\} \quad \text{a.e.~in } \Lambda
	\end{aligned}
	\]
	and denote them by $\mu_a$ and $\mu_b$, too, with a little abuse of notation.
	Then, by means of standard arguments as \eg in~\cite[Thm.~2.29]{Troe10}, one deduces 
	\eqref{eq:gradeq3},~\eqref{eq:compla2}, and~\eqref{eq:complb2} from~\eqref{eq:vi2}.
	Finally,~\eqref{eq:compllambda} follows from 
	$\lambda \in \cone(\{ \mathrm{e}_i : i\in \AA\})$ and $(Gu - b)_i = 0$ for all~$i\in \AA$.
\end{proof}

\begin{remark}
	Theorem~\ref{thm:lagrange} readily carries over to vector valued box constraints in~$\R^n$ 
	as in~\eqref{eq:PCk}, but in order to keep the discussion concise, we restricted it to the scalar case here.
	
\end{remark}

\fontsize{9}{10.5}\selectfont
\bibliographystyle{siam}
\bibliography{reference}
\end{document}